\documentclass[a4paper,11pt]{article}

\usepackage{amsmath}
\usepackage{amssymb}
\usepackage{latexsym}
\usepackage[dvips]{graphicx}
\usepackage{psfrag}

\newtheorem{theorem}{Theorem}[section]
\newtheorem{lemma}{Lemma}[section]
\newtheorem{prop}{Proposition}[section]
\newtheorem{definition}{Definition}[section]
\newtheorem{remark}{Remark}

\numberwithin{equation}{section}

\newenvironment{proof}{\noindent{\textsc{Proof.}}}
{$\hfill\Box$\vspace{0.1 cm}\\}

\newcommand{\R}{\mathbb R}
\newcommand{\tv}{\hbox{Tot.Var.}}
\newcommand{\N}{{\mathbb N}}

\newcommand{\abs}[1]{\left\vert #1 \right\vert}

\newcommand{\Rsol}{\mathcal{RS}}

\newcommand{\pt}{\partial_t}
\newcommand{\px}{\partial_x}
\newcommand{\sx}{( \rho^l, v^l )}
\newcommand{\dx}{( \rho^r, v^r )}
\newcommand{\dom}{\mathcal D}
\newcommand{\ddom}{\mathcal D_{v_1,v_2,w_1,w_2}}
\newcommand{\rs}[1]{\Rsol^q_{#1}((\rho^l, v^l), (\rho^r, v^r))}

\newcommand{\be}{\begin{equation}}
\newcommand{\ee}{\end{equation}}
\newcommand{\bu}{{\bf u}}

\begin{document}

%\date{}

\title{The Aw-Rascle traffic model with locally constrained flow}

\author{M.~Garavello\thanks{DiSTA, Universit\`a del Piemonte Orientale
    ``A. Avogadro'', Alessandria, Italy.
    E-mail: \texttt{mauro.garavello@mfn.unipmn.it}. Partially supported
    by Dipartimento di Matematica e Applicazioni of the University of
    Milano-Bicocca.} \and
  P.~Goatin\thanks{IMATH, Universit\'e du Sud Toulon-Var,
    France. E-mail: \texttt{goatin@univ-tln.fr}.}}

\maketitle

\begin{abstract}
  We consider solutions of the Aw-Rascle model for traffic flow fulfilling a constraint on the flux at $x=0$. 
  Two different kinds of solutions are proposed:
  at $x=0$ the first one conserves both the number of vehicles and the
  generalized momentum, while the second one conserves only the number of
  cars.
  We study the invariant domains for these solutions and we compare
  the two Riemann solvers in terms of total variation of relevant quantities.
  Finally we construct {\it ad hoc} finite volume numerical schemes to compute these solutions.
\end{abstract}

%\begin{keywords} 

%\begin{AMS}
%\end{AMS}

\section{Introduction}\label{se:introduction}

The paper deals with solutions to the Aw-Rascle vehicular traffic model~\cite{Aw-Rascle_2000}
\begin{equation}
  \label{eq:aw-rascle2}
  \left\{
    \begin{array}{l}
      \pt \rho + \px (\rho v) = 0,\\
      \pt y + \px \left(y v\right) = 0,\\
    \end{array}
  \right.
\end{equation}
satisfying a constraint on the first component of the flux at $x=0$:
\begin{equation}
  \label{eq:constraint}
  \rho(t,0) v(t,0) \le q,
\end{equation}
where $q > 0$ is a given constant.
Here $\rho$, $v$ and $y$
denote respectively the density, the average speed and a generalized
momentum of cars in a road. Moreover $y = \rho \left(v+p(\rho) \right)$,
where $p \in C^2([0,+\infty[; [0,+\infty[)$ is a pressure function satisfying
\begin{equation}
  \label{eq:pressure-function}
  \left\{
    \begin{array}{l}
      p(0)=0, \\
      p'( \rho) > 0 \textrm{ for every } \rho > 0,\\
      \rho \mapsto \rho p(\rho) \textrm{ is strictly convex.}
    \end{array}
  \right.
\end{equation}
% \begin{enumerate}
% \item $p \in C^2([0,+\infty[)$;
%
% \item $p(0)=0$;
%
% \item $p'( \rho) > 0$ for every $\rho > 0$;
%
% \item $\rho \mapsto \rho p(\rho)$ is strictly convex.
% \end{enumerate}
%
Problem~\eqref{eq:aw-rascle2},~\eqref{eq:constraint} models the presence of a 
constraint on traffic flow at the point $x=0$, such as a toll gate, a traffic light, 
a construction site, etc. All these situations limit the flow at a specific location 
along the road. 
Conservation laws with unilateral constraints as~\eqref{eq:constraint}
were first introduced in~\cite{MR2300671}, see also~\cite{AGS, ColomboGoatinRosini,mtq} 
for further analytical results and applications. In these papers, the scalar 
Lighthill-Whitham~\cite{Lighthill-Whitham_1955} and Richards~\cite{Richards_1956} traffic model is coupled with
a (possibly time-dependent) constraint on the flow, as in~\eqref{eq:constraint}.

The model presented here constitutes the first example of a system
of two equations with constrained flux. 
The Aw-Rascle model~\eqref{eq:aw-rascle2}
belongs to the so-called ``second order'' traffic models,
i.e. models consisting in two equations
(see~\cite{ColomboNewmod, payne_1971, whitham_1974} for other examples). 
System~\eqref{eq:aw-rascle2} can also be written
\begin{equation}
  \label{eq:aw-rascle1}
  \left\{
    \begin{array}{l}
      \pt \rho + \px (\rho v) = 0,\\
      \pt (\rho (v+p(\rho))) + \px (\rho v (v+p(\rho))) = 0.\\
    \end{array}
  \right.
\end{equation}
The first equation in~\eqref{eq:aw-rascle1} states the conservation 
of the number of vehicles, moving with flow rate $\rho v$.
The second equation is derived from the former one and from
the evolution equation of the quantity $w=v+p(\rho)$
(often referred to as ``Lagrangian marker''),
which moves with velocity $v$:
$$
 \pt (v+p(\rho)) + v \px (v+p(\rho)) = 0.
$$
The system in conservative form~\eqref{eq:aw-rascle1}
belongs to the Temple class~\cite{Temple}, i.e. systems for
which shock and rarefaction curves in the unknowns' space coincide.
In particular, for such systems the interaction of two waves 
of the same family can only give rise to a wave of the same family. 

The Aw-Rascle model~\eqref{eq:aw-rascle1} has been widely studied 
in the mathematical literature. Concerning the model itself, 
various extensions have been proposed,
see~\cite{MR2366138, MR2438216, goatin_phase_2006, 
  Greenberg_2001-02, MR2366992}.
The model can also be used to describe traffic flow on a road network, 
as explained in~\cite{garavello-piccoli_ARModel_2006, MR2223072, MR2237163}.

In this paper we restrict the analysis to the
Riemann problem for~\eqref{eq:aw-rascle2},
\eqref{eq:constraint}, i.e. to the Cauchy problem with piecewise
constant initial data of the form
$$
(\rho, y)(0,x) = \left\{
        \begin{array}{ll}
          (\rho^l, y^l), & \textrm{ if } x<0,\\
          (\rho^r, y^r), & \textrm{ if } x>0.
        \end{array}
      \right.
$$
We propose two Riemann solvers, described in Sections 2.1 and 2.2:
the first one conserves at $x=0$ both the number of cars and the generalized
momentum, while the second one does not conserve the generalized momentum.
In particular, the first Riemann solver produces a non-entropic shock
wave at $x=0$, which travels with zero velocity.
In Section 3 we describe the invariant domains corresponding to the 
two Riemann solvers, and in Section 4 we compare
the total variation of relevant 
quantities. Section 5 is devoted to the construction
of {\it ad hoc} numerical schemes
designed to capture the proposed solutions.

%
%
% Riemann problem
%
%
\section{The Riemann problem}
\label{se:Riemann_problem}

In this section we deal with the Riemann problem
\begin{equation}
  \label{eq:Riemann_problem}
  \left\{
    \begin{array}{l}
      \pt \rho + \px (\rho v) = 0,\\
      \pt (\rho (v+p(\rho))) + \px (\rho v (v+p(\rho))) = 0,\\
      (\rho, v)(0,x) = \left\{
        \begin{array}{ll}
          (\rho^l, v^l), & \textrm{ if } x<0,\\
          (\rho^r, v^r), & \textrm{ if } x>0,
        \end{array}
      \right.
    \end{array}
  \right.
\end{equation}
in the domain $\dom=\R^+\times\R^+$,
%$(\rho^l, v^l) \in \dom$ and
%$(\rho^r, v^r) \in \dom$,
and with the constraint~\eqref{eq:constraint}.

We denote by $f(\rho,v)$ the flux for system~(\ref{eq:aw-rascle1}),
and with $f_1(\rho,v)$, $f_2(\rho,v)$ its components, i.e.
\begin{equation}
  \label{eq:flux}
  f (\rho, v) = \left(
    \begin{array}{c}
      f_1(\rho, v) \\
      f_2(\rho, v)
    \end{array}
  \right) = \left(
    \begin{array}{c}
      \rho v \\
      \rho v ( v + p(\rho))
    \end{array}
  \right).
\end{equation}
For reader's comfort, we resume in the following tables 
the relevant quantities concerning systems \eqref{eq:aw-rascle2}, 
\eqref{eq:aw-rascle1} respectively.
In $(\rho, y)$ plane they write:
\begin{displaymath}
  % \label{eq:quantities_RHO_y_plane}
  \begin{array}{ll}
    \lambda_1 = -p(\rho) + \frac{y}{\rho} - \rho p'(\rho)
    &
    \lambda_2 = -p(\rho) + \frac{y}{\rho} \\
    r_1 = \left(
      \begin{array}{c}
        -1 \\ -\frac{y}{\rho}
      \end{array}
    \right)
    &
    r_2 = \left(
      \begin{array}{c}
        1 \\ \frac{y}{\rho} + \rho p'(\rho)
      \end{array}
    \right)
    \\
    \nabla \lambda_1 \cdot r_1 = 2 p'(\rho) + \rho p''(\rho) > 0
    &
    \nabla \lambda_2 \cdot r_2 = 0
    \\
    L_1 (\rho; \rho_0, y_0) = \frac{y_0}{\rho_0} \rho
    &
    L_2 (\rho; \rho_0, y_0) = \frac{y_0}{\rho_0} \rho + \rho \left(p(\rho) -
      p(\rho_0)\right)\\
    z = \frac{y}{\rho} -p(\rho) & w = \frac{y}{\rho}
  \end{array}
\end{displaymath}
In $(\rho, v)$ plane their expression is:
\begin{displaymath}
  % \label{eq:quantities_rho_v_plane}
  \begin{array}{ll}
    \lambda_1 = v - \rho p'(\rho)
    &
    \lambda_2 = v \\
    r_1 = \left(
      \begin{array}{c}
        -1 \\ p'(\rho)
      \end{array}
    \right)
    &
    r_2 = \left(
      \begin{array}{c}
        1 \\ 0
      \end{array}
    \right)
    \\
    \nabla \lambda_1 \cdot r_1 = 2 p'(\rho) + \rho p''(\rho) > 0
    &
    \nabla \lambda_2 \cdot r_2 = 0
    \\
    L_1 (\rho; \rho_0, v_0) = v_0 +p(\rho_0) - p(\rho)
    &
    L_2 (\rho; \rho_0, v_0) = v_0
    \\
    z = v & w = v + p(\rho)
  \end{array}
\end{displaymath}
Above, $\lambda_1$ and $\lambda_2$ denote the eigenvalues
of the Jacobian matrix $Df$,
$r_1$ and $r_2$ the corresponding right eigenvectors,
$L_1$ and $L_2$ the first and the second
Lax curve, $z$ and $w$ the $1$- and $2$-Riemann invariant respectively.

We remark that the system is strictly hyperbolic away from $\rho=0$
(i.e. $\lambda_1 < \lambda_2$). Moreover
the first characteristic speed is genuinely nonlinear,
with characteristic speed that can change sign, and the 
second one is linearly degenerate with strictly positive speed.
\begin{definition}
  A Riemann solver for system~(\ref{eq:Riemann_problem}) is a function,
  which associates, for every initial condition $(\rho^l, v^l) \in \dom$,
  $(\rho^r, v^r) \in \dom$, a map belonging to $L^1(\R)$ and representing a
  solution to~(\ref{eq:Riemann_problem}) at time $t = 1$.
\end{definition}

By $\Rsol$ we denote the classical Riemann solver
for~(\ref{eq:Riemann_problem}), i.e. the Riemann solver without
the constraint~(\ref{eq:constraint}); see for example~\cite{Aw-Rascle_2000}.
We introduce some more notation.

Given $(\rho^l, v^l) \in \dom$ and $q>0$, let us consider the set
\begin{eqnarray}
  \label{eq:I_1}
  I_1 & = & \left\{ \rho \in [0,+\infty[ \ \colon \rho L_1(\rho; \rho^l, v^l) =q
  \right\} \\
  & = & \left\{ \rho \in [0,+\infty[ \ 
    \colon \rho ( v^l + p(\rho^l) - p(\rho)) =q
  \right\}. \nonumber
\end{eqnarray}
The set $I_1$ contains the densities of all the points
$(\rho, v) \in \dom$ belonging to the Lax curve of the first family
passing through $(\rho^l, v^l)$ and such that $f_1(\rho^l, v^l) = q$.
If $I_1 \ne \emptyset$, then we denote by $\hat \rho$, $\hat v$,
$\check \rho_1$, $\check v_1$ respectively
\begin{equation}
  \label{eq:hat_check}
  \hat \rho = \max I_1, \quad
  \hat v = \frac{q}{\hat \rho}, \quad
  \check \rho_1 = \min I_1, \quad
  \check v_1 = \frac{q}{\check \rho_1};
\end{equation}
see Figure~\ref{fig:I_1}.
\begin{figure}
  \centering
  \begin{psfrags}
    \psfrag{l}{ \hspace{-.6cm}$(\rho^l, \rho^l v^l)$}
    \psfrag{w}{ $L_1$}
    \psfrag{q}{ $q$}
    \psfrag{crho}{ $\check \rho_1$}
    \psfrag{hrho}{ $\hat \rho$}
    \psfrag{vrho}{\hspace{-.5cm} $f_1$}
    \psfrag{rho}{ $\rho$}
    \includegraphics[width = 11cm]{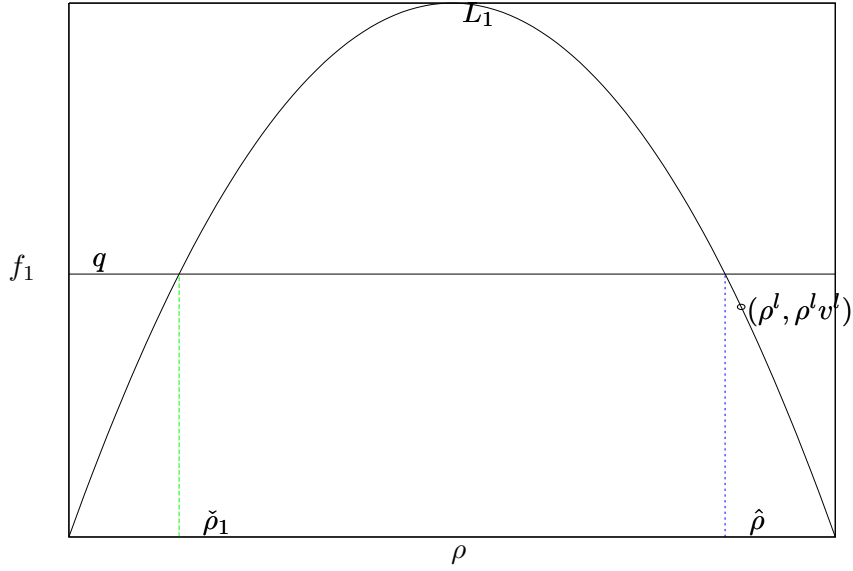}
  \end{psfrags}
  \caption{The set $I_1$ and the quantities of equation~(\ref{eq:hat_check}).}
  \label{fig:I_1}
\end{figure}
Given $(\rho^r, v^r) \in \dom$ and $q>0$, let $\check \rho_2$ and
$\check v_2$ be defined by
\begin{equation}
  \label{eq:check_rho1}
 % \check \rho_1 \in [0, +\infty [ \ , ~\hbox{s.t.}~
  \check \rho_2 L_2(\check \rho_2; \rho^r, v^r) = q, \quad
  \check v_2 = \frac{q}{\check \rho_2};
\end{equation}
i.e. $(\check\rho_2, \check v_2)$ belongs to the Lax curve of the second
family passing through $(\rho^r, v^r)$ and satisfies
$f_1(\check\rho_2, \check v_2) = q$.
In particular, note that $\check v_2 =v^r$ and $\check \rho_2 = q/v^r$. \\
Given $(\rho^l, v^l)$ and $(\rho^r, v^r) \in \dom$, %and $(\check \rho_2, \check v_2) \in \dom$,
let us consider the set
\begin{eqnarray}
  \label{eq:I_2}
  I_2 & = & \left\{ \rho \in [0,+\infty [ \ \colon 
   L_1 (\rho;  \rho^l, v^l) =  L_2(\rho; \rho^r, v^r)
  \right\} \\
  & = & \left\{ \rho \in [0,+\infty [ \ \colon 
  v^l + p(\rho^l) - p(\rho) =v^r
  \right\} \nonumber
\end{eqnarray}
and define
\begin{equation}
  \label{eq:rho2m}
  %\rho_2^m \in [0, \rho_{max}], \quad
  \rho^m = \max I_2, \quad
  v^m = v^r,
\end{equation}
which provide the intermediate state for the classical solution
to~(\ref{eq:Riemann_problem}).

\begin{lemma}
  \label{lem:I1_not_empty}
  Let $(\rho^l, v^l), (\rho^r, v^r) \in \dom$ and $q>0$ be fixed.
  Assume~(\ref{eq:pressure-function}) holds.
  If
  \begin{displaymath}
    f_1(\Rsol((\rho^l, v^l),(\rho^r, v^r))(0)) > q,
  \end{displaymath}
  then the set $I_1$ is not empty and it consists in exactly two
  different points: $I_1=\left\{ \check\rho_1, \hat\rho  \right\}$.
\end{lemma}

\begin{proof}
  Notice that the function $\rho \mapsto \rho L_1(\rho; \rho^l, v^l)$
  is strictly concave by the hypotheses~(\ref{eq:pressure-function})
  on the pressure function $p(\rho)$
  and so, by~(\ref{eq:I_1}), the cardinality of $I_1$ is at most $2$. \\
  Denote with $(\rho^M, v^M)$ the trace of
  $\Rsol((\rho^l, v^l), (\rho^r, v^r))$ at $x=0+$. Since the waves of
  the second family have strictly positive speed, then
  $v^M = L_1(\rho^M; \rho^l, v^l)$.
  Therefore, if $I_1 = \emptyset$ or it contains only one element,
  then $\rho^M v^M \le q$ and therefore
  $f_1(\Rsol((\rho^l, v^l),(\rho^r, v^r))(0)) \le q$, which is a contradiction.
  Thus the only possibility is that $I_1$ is composed exactly by
  two elements.
\end{proof}

We propose two different ways of solving
problem~(\ref{eq:Riemann_problem})-(\ref{eq:constraint}).

%
% The first Riemann solver
%
\subsection{The Constrained Riemann Solver $\Rsol_{1}^q$}
\label{sse:RS1}

In this part, we introduce the Riemann solver $\Rsol_1^q$
for~(\ref{eq:Riemann_problem})-(\ref{eq:constraint}),
which
is characterized by the conservation of both the quantities
$\rho$ and $y = \rho (v + p(\rho))$ at $x=0$.

Fix $(\rho^l, v^l) , (\rho^r, v^r) \in \dom$. The Riemann solver
$\Rsol_1^q$ is defined as follows.

\begin{enumerate}
\item If $f_1(\Rsol \left(\sx, \dx \right)(0)) \le q$, then % we put
  \begin{equation}
    \label{eq:rs1_case1}
    \Rsol^q_1 \left( \sx, \dx \right)(x) = \Rsol \left( \sx, \dx \right)(x)
  \end{equation}
  for every $x \in \R$.

\item If $f_1(\Rsol \left(\sx, \dx \right)(0)) > q$, then
  \begin{equation}
    \label{eq:rs1_case2}
    \Rsol^q_1 \left( \sx, \dx \right) (x) = \left\{
      \begin{array}{ll}
        \Rsol \left( \sx, (\hat \rho, \hat v)
        \right) (x), & \textrm{ if } x<0,\\
        \Rsol \left( (\check \rho_1, \check v_1 )
          , \dx \right) (x), 
        & \textrm{ if } x>0.
      \end{array}
    \right.
  \end{equation}

\end{enumerate}

\begin{prop}
  The Riemann solver $\Rsol^q_1$ satisfies
  \begin{displaymath}
    f_1(\Rsol^q_1\left(\sx, \dx \right)(0)) \le q
  \end{displaymath}
  for every $\sx$ and $\dx$. 
\end{prop}
The proof follows directly from the construction of the Riemann
solver $\Rsol^q_1$.

\begin{remark}
  The Riemann solver $\Rsol^q_1$ is determined by imposing the
  conservation of both quantities $\rho$ and $y$ at $x=0$ and
  by respecting the constraint condition~(\ref{eq:constraint});
  see also~\cite{MR2223072} for an example of a Riemann solver
  at a node, which conserves both $\rho$ and $y$.
\end{remark}

In the following, we denote by $w(\Rsol^q_1\left(\sx, \dx \right)(x))$  the $w$ component 
of the Riemann solver $\Rsol^q_1\left(\sx, \dx \right)(x)$.
\begin{prop}
  The Riemann solver $\Rsol^q_1$ satisfies the maximum principle 
  on the second Riemann invariant $w=v+p(\rho)$, i.e.
  \begin{displaymath}
    \min\left\{ w^l,w^r \right\}  \le 
    w(\Rsol^q_1\left(\sx, \dx \right)(x)) \le 
    \max\left\{ w^l,w^r \right\}, \quad \forall x \in \R.
  \end{displaymath}
%  where $w(\Rsol^q_1\left(\sx, \dx \right)(x))$ denotes the Riemann
%  coordinate $w$ of the Riemann solver $\Rsol^q_1\left(\sx, \dx \right)(x)$.
\end{prop}

% Above, $w(\Rsol^q_1\left(\sx, \dx \right)(\cdot))$ denotes
% the $w$ component of $\Rsol^q_1$.
The property easy follows from the maximum principle
satisfied by the classical Riemann solvers
$ \Rsol \left( \sx, (\hat \rho, \hat v) \right)$ for $x<0$ and
$ \Rsol \left( (\check \rho_1, \check v_1), \dx \right)$ for $x>0$,
and by the fact that $\hat w= \check w_1=w^l$.

%
% The second Riemann solver
%
\subsection{The constrained Riemann solver $\Rsol^q_{2}$}
\label{sse:RS2}

In this part we describe the Riemann solver $\Rsol^q_2$, 
which conserves only the car density $\rho$
at $x=0$.

Fix $(\rho^l, v^l) \in \dom$, $(\rho^r, v^r) \in \dom$. The Riemann solver
$\Rsol^q_2$ is defined as follows.

\begin{enumerate}
\item If $f_1(\Rsol \left(\sx, \dx \right)(0)) \le q$, then we put
  \begin{equation}
    \label{eq:rs2_case1}
    \Rsol^q_2 \left( \sx, \dx \right) (x) = \Rsol \left( \sx, \dx \right) (x)
  \end{equation}
  for every $x \in \R$.

\item If $f_1(\Rsol \left(\sx, \dx \right)(0)) > q$, then
  \begin{equation}
    \label{eq:rs2_case2}
    \Rsol^q_2 \left( \sx, \dx \right) (x) = \left\{
      \begin{array}{ll}
        \Rsol \left( \sx, (\hat \rho, \hat v
          ) \right) (x), & \textrm{ if } x<0,\\
        \Rsol \left( (\check \rho_2, \check v_2
          ), \dx \right) (x), 
        & \textrm{ if } x>0.
      \end{array}
    \right.
  \end{equation}
\end{enumerate}

\begin{prop}
  The Riemann solver $\Rsol^q_2$ satisfies
  \begin{displaymath}
    f_1(\Rsol^q_2(\sx, \dx)(0)) \le q
  \end{displaymath}
  for every $\sx$ and $\dx$. 
\end{prop}

The proof follows directly from the construction of the Riemann
solver $\Rsol^q_2$.

\begin{remark}
  The Riemann solver $\Rsol^q_2$ conserves only the density at $x=0$;
  therefore it is in the same spirit of Riemann solvers introduced
  for traffic at junctions in~\cite{garavello-piccoli_ARModel_2006}.
\end{remark}

%
%
% Invariant domains
%
%
\section{Invariant domains for $\Rsol^q_1$ and $\Rsol^q_2$}
\label{se:invariant domains}

In this section, we want to describe the invariant regions for the
Aw-Rascle system with constraints. First, we recall that,
for every $0 < v_1 < v_2$, $0 < w_1 < w_2$ and $v_2 < w_2$, the set
\begin{equation}
  \label{eq:invariant_no_constraint}
  \ddom = \left\{
    (\rho, v) \in \dom: v_1 \le v \le v_2, \, w_1 \le v + p(\rho) \le w_2
  \right\}
\end{equation}
is invariant for~(\ref{eq:aw-rascle2}); see Figure~\ref{fig:invariant_domain}
and~\cite{MR784005}.
The hypothesis $v_2 < w_2$ implies that the Riemann invariants
$w = w_2$ and $z = v_2$ intersect in $\dom$ at a point different from
the origin.
For a given $q > 0$, we define the function of class $C^2(]0, +\infty[)$
\begin{equation}
  \label{eq:h_q}
  \begin{array}{rcc}
    h_q :\, ]0, +\infty[ & \longrightarrow & \R \\
    v & \longmapsto & p \left( \frac{q}{v} \right) + v,
  \end{array}
\end{equation}
which gives the value of the Riemann invariant $w$
of the point $(\tilde \rho, v) \in \dom$ such that $\tilde \rho v = q$.
Indeed we have that $h_q(v) = w$ if and only if $w = v + p(\rho)$ with
$\rho v = q$.

\begin{lemma}
  \label{le:h_q}
  Fix $q>0$ and assume~(\ref{eq:pressure-function}).
  There exists $\bar v = \bar v (q)> 0$ such that the function
  $h_q(v)$ is strictly decreasing in $]0, \bar v[$ and
  strictly increasing in $]\bar v, +\infty[$.
\end{lemma}

\begin{proof}
  We have
  \begin{displaymath}
    h_q''(v) = \frac{q}{v^3} \left[
      2 p'\left( \frac{q}{v} \right) + 
      \frac{q}{v} p''\left( \frac{q}{v} \right)
    \right]
  \end{displaymath}
  and so, by~(\ref{eq:pressure-function}), we deduce that
  $h_q''(v) > 0$ for every $v>0$; this means that $h_q'(v)$ is a
  strictly increasing function.
  Note also that~(\ref{eq:pressure-function}) implies that
  \begin{equation}
    \label{eq:pressure-infinity}
    \lim_{\rho \to +\infty} p(\rho) = + \infty.
  \end{equation}
  Indeed, if~(\ref{eq:pressure-infinity}) does not hold, then
  there exists $M > 0$ such that $p(\rho) \le M$
  and so $\rho p(\rho) \le M \rho$ for every
  $\rho >0$. This is not possible since the map
  $\rho \mapsto \rho p(\rho)$ is strictly convex.
  This implies that
  \begin{displaymath}
    \lim_{v \to 0^+} h_q(v) = +\infty \quad 
    \textrm{ and }\quad
    \lim_{v \to +\infty} h_q'(v) = 1;
  \end{displaymath}
  hence, since $h'_q$ is a strictly increasing function, 
  there exists a unique $\bar v > 0$ such that $h_q'(\bar v) = 0$.
  Therefore $h_q$ is strictly decreasing in $]0, \bar v[$ and
  strictly increasing in $]\bar v, +\infty[$.
\end{proof}

\begin{figure}
  \centering
  \begin{psfrags}
    \psfrag{v}{$v_1$} \psfrag{vv}{$v_2$}
    \psfrag{w}{\hspace{.8cm}$w_1$} \psfrag{ww}{\hspace{.8cm}$w_2$}
    \psfrag{vrho}{\hspace{-.2cm}$\rho v$} \psfrag{rho}{$\rho$}
    \includegraphics[width=8cm]{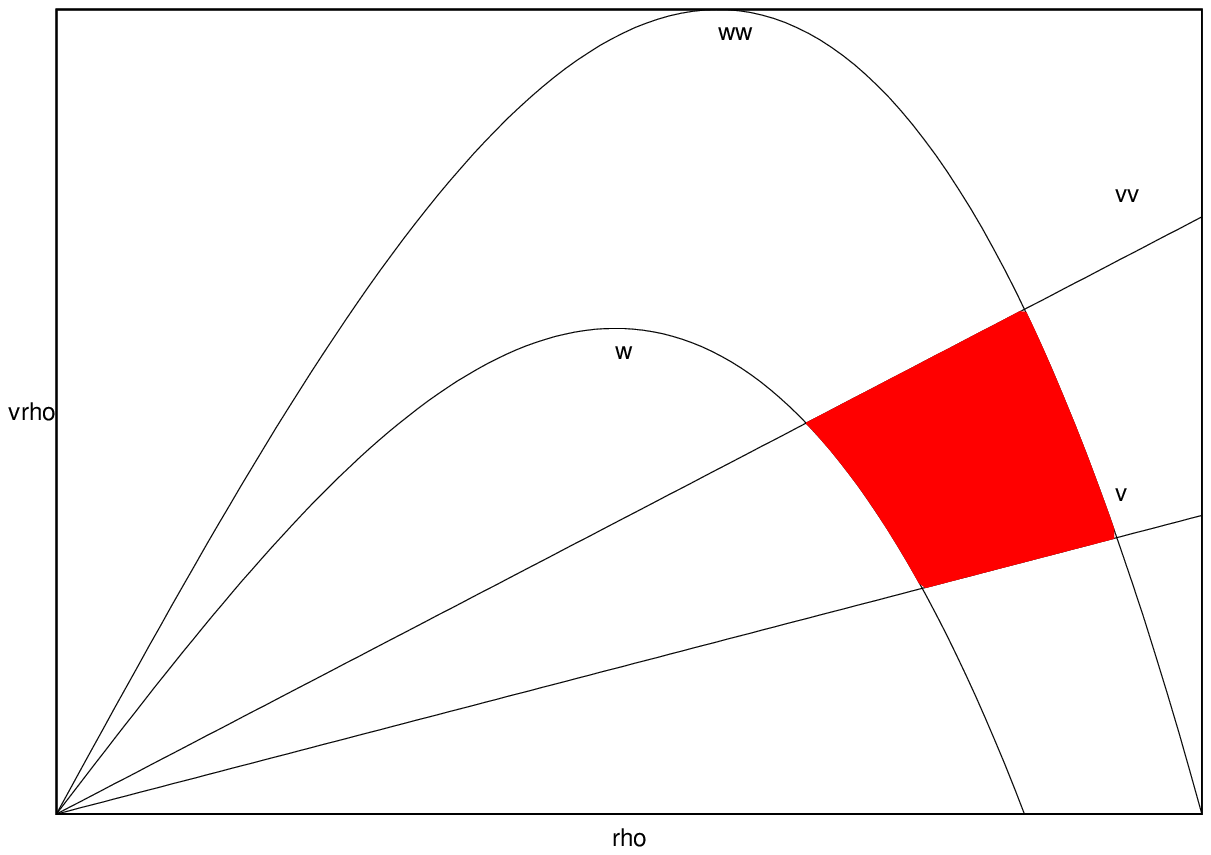}
  \end{psfrags}
  \caption{The invariant domain $\ddom$.}
  \label{fig:invariant_domain}
\end{figure}

\begin{prop}
  Fix $0 < v_1 < v_2$, $0 < w_1 < w_2$, $v_2 < w_2$ and $q > 0$.
  If $h_q(v) \ge w_2$ for every $v \in [v_1, v_2]$, then $\ddom$
  is invariant for both the Riemann solvers $\Rsol^q_1$ and $\Rsol^q_2$. 
\end{prop}

\begin{proof}
  The hypothesis $h_q(v) \ge w_2$ for every $v \in [v_1, v_2]$ implies
  that
  \begin{displaymath}
    \sup \left\{
      f_1 (\rho,v) : (\rho,v) \in \ddom, v \in [v_1, v_2], v+p(\rho)= w_2
    \right\} \le q
  \end{displaymath}
  and so
  \begin{displaymath}
    \sup \left\{
      f_1 (\rho,v) : (\rho,v) \in \ddom  \right\} \le q.
  \end{displaymath}
  Therefore the Riemann solvers $\Rsol^q_1$ and $\Rsol^q_2$
  in the domain $\ddom$ coincide with $\mathcal{RS}$.
\end{proof}

%
% rs1
%

\subsection{The Riemann solver $\Rsol^q_1$}

The next proposition describes the invariant domains for $\Rsol^q_1$.

\begin{prop}
  \label{prop:rs2_cns}
  Fix $0 < v_1 < v_2$, $0 < w_1 < w_2$, $v_2 < w_2$ and $q > 0$.
  Assume~(\ref{eq:pressure-function}) and 
  that there exists $\bar v \in [v_1, v_2]$ such that
  $h_q(\bar v) < w_2$.
  The set
  $\ddom$ is invariant for the Riemann solver $\Rsol^q_1$,
  if and only if
  \begin{equation}
    \label{eq:cns-rs2}
    h_q(v_1) \ge w_2 \quad \textrm{ and } \quad h_q(v_2) \ge w_2;
  \end{equation}
  see Figure~\ref{fig:inv_domain_rs2}.
\end{prop}

\begin{proof}
  Clearly, if condition~(\ref{eq:cns-rs2}) holds, then
  the set $\ddom$ is invariant for $\Rsol^q_1$, since both
  $(\hat \rho, \hat v)$ and $(\check \rho_1, \check v_1)$
  belong to $\ddom$ for every possible choice of initial conditions in
  $\ddom$.

  Assume now that $\ddom$ is invariant for $\Rsol^q_1$.\\
  If $h_q (v_1) < w_2$, then
  denote with $(\rho^l, v^l) = (\rho^r, v^r) \in \ddom$ the
  solution to the system
  \begin{displaymath}
    \left \{
      \begin{array}{l}
        v^l + p(\rho^l) = w_2, \\
        v^l = v_1.
      \end{array}
    \right.
  \end{displaymath}
  By hypotheses, we deduce that $\rho^l v^l > q$ and so
  the trace of the Riemann solver $\rs1$ at the point $0-$ is given by
  $(\hat \rho, \hat v)$, which does not belong to
  $\ddom$, since $h_q(v_1) < w_2$. This argument shows that
  $h_q (v_1) \ge w_2$.\\
  If $h_q (v_2) < w_2$, then
  denote with $(\rho^l, v^l) = (\rho^r, v^r) \in \ddom$ the
  solution to the system
  \begin{displaymath}
    \left \{
      \begin{array}{l}
        v^l + p(\rho^l) = w_2, \\
        v^l = v_2.
      \end{array}
    \right.
  \end{displaymath}
  By hypotheses, we deduce that $\rho^l v^l > q$ and so
  the trace of the Riemann solver $\rs1$ at the point $0+$ is given by
  $(\check \rho_1, \check v_1)$, which does not belong to
  $\ddom$, since $h_q(v_2) < w_2$. This argument shows that
  $h_q (v_2) \ge w_2$. This completes the proof.
\end{proof}

\begin{figure}
  \centering
  \begin{psfrags}
    \psfrag{v}{$v_1$} \psfrag{vv}{$v_2$}
    \psfrag{w}{\hspace{.8cm}$w_1$} \psfrag{ww}{\hspace{.8cm}$w_2$}
    \psfrag{vrho}{\hspace{-.2cm}$\rho v$} \psfrag{rho}{$\rho$}
    \psfrag{q}{$q$}
    \includegraphics[width=8cm]{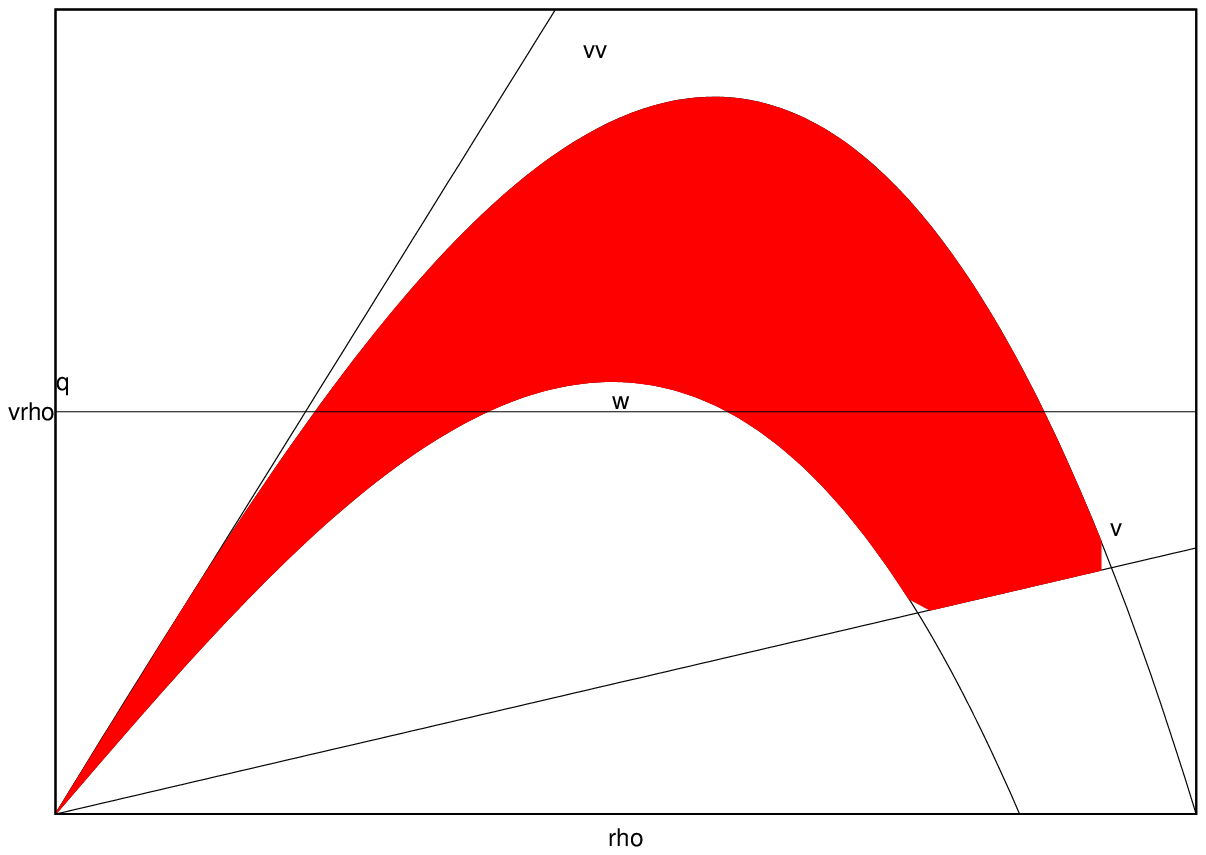}
  \end{psfrags}
  \caption{The invariant domain $\ddom$ for the Riemann solver
    $\mathcal{RS}_1^q$.}
  \label{fig:inv_domain_rs2}
\end{figure}

%
%
%rs2
%
%
\subsection{The Riemann solver $\Rsol^q_2$}

In this part, we describe the invariant domains for $\Rsol^q_2$.
First let us introduce the following necessary conditions.
\begin{lemma}
  \label{le:rs1_cn1}
  Fix $0 < v_1 < v_2$, $0 < w_1 < w_2$, $v_2 < w_2$ and $q > 0$.
  Assume~(\ref{eq:pressure-function}) and
  that there exists $\bar v \in [v_1, v_2]$ such that
  $h_q(\bar v) < w_2$.
  If the set
  $\ddom$ is invariant for the Riemann solver $\Rsol^q_2$,
  then $h_q(v_1) \ge w_2$.
\end{lemma}

\begin{proof}
  Assume by contradiction that $h_q(v_1) < w_2$.
  Denote $(\rho^l, v^l) = (\rho^r, v^r) \in \ddom$ the
  solution to the system
  \begin{displaymath}
    \left \{
      \begin{array}{l}
        v^l + p(\rho^l) = w_2, \\
        v^l = \bar v.
      \end{array}
    \right.
  \end{displaymath}
  By hypotheses, we deduce that $\rho^l v^l > q$ and so
  the trace of the Riemann solver $\rs2$ at the point $0-$ is given by
  $(\hat \rho, \hat v)$, which does not belong to
  $\ddom$, since $h_q(v_1) < w_2$.
\end{proof}

\begin{lemma}
  \label{le:rs1_cn2}
  Fix $0 < v_1 < v_2$, $0 < w_1 < w_2$, $v_2 < w_2$ and $q > 0$.
  Assume~(\ref{eq:pressure-function}) and
  that there exists $\bar v \in [v_1, v_2]$ such that
  $h_q(\bar v) < w_2$.
  If the set
  $\ddom$ is invariant for the Riemann solver $\Rsol^q_2$,
  then $h_q(v) \ge w_1$ for every $v \in [v_1,v_2]$.
\end{lemma}

\begin{proof}
  Assume by contradiction that $h_q(\tilde v) < w_1$ for some
  $\tilde v \in [v_1, v_2]$.
  Denote $(\rho^l, v^l) = (\rho^r, v^r) \in \ddom$ the
  solution to the system
  \begin{displaymath}
    \left \{
      \begin{array}{l}
        v^l + p(\rho^l) = w_2, \\
        v^l = \tilde v.
      \end{array}
    \right.
  \end{displaymath}
  By hypotheses, we deduce that $\rho^l v^l > q$ and so
  the trace of the Riemann solver $\rs2$ at the point $0+$ is given by
  $(\check \rho_2, \check v_2)$, which does not belong to
  $\ddom$, since $\check v_2 = \tilde v$ and $h_q(\tilde v) < w_1$.
\end{proof}

\begin{figure}
  \centering
  \begin{psfrags}
    \psfrag{v}{$v_1$} \psfrag{vv}{$v_2$}
    \psfrag{w}{\hspace{.8cm}$w_1$} \psfrag{ww}{\hspace{.8cm}$w_2$}
    \psfrag{vrho}{\hspace{-.2cm}$\rho v$} \psfrag{rho}{$\rho$}
    \psfrag{q}{$q$}
    \includegraphics[width=8cm]{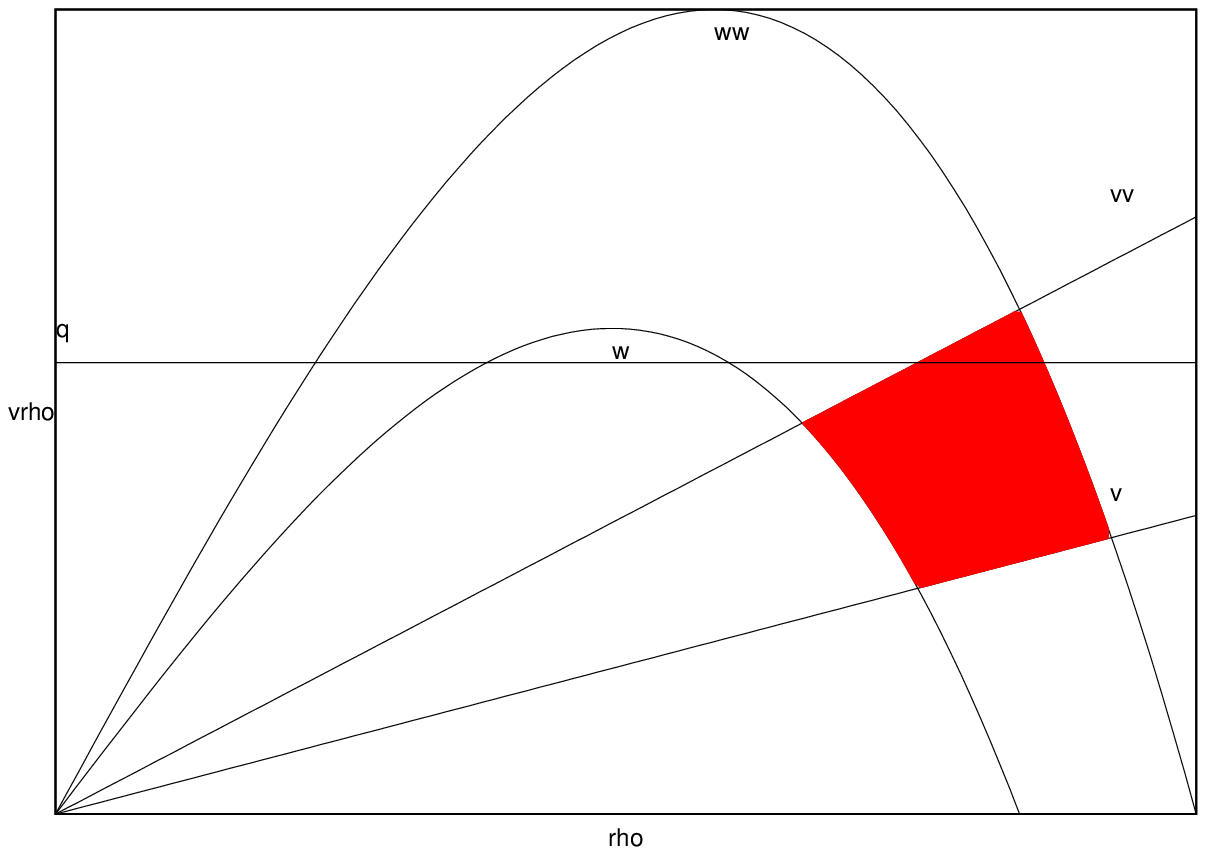}
  \end{psfrags}
  \caption{The invariant domain $\ddom$ for the Riemann solver
    $\mathcal{RS}_2^q$.}
  \label{fig:inv_domain_rs1}
\end{figure}

We have the following proposition about necessary and sufficient
conditions for a domain to be invariant for $\Rsol^q_2$.

\begin{prop}
  Fix $0 < v_1 < v_2$, $0 < w_1 < w_2$, $v_2 < w_2$ and $q > 0$.
  Assume~(\ref{eq:pressure-function}) and
  that there exists $\bar v \in [v_1, v_2]$ such that
  $h_q(\bar v) < w_2$.
  The set
  $\ddom$ is invariant for the Riemann solver $\Rsol^q_2$ 
  (see Figure~\ref{fig:inv_domain_rs1})
  if and only if 
  \begin{equation}
    \label{eq:inv_cond_rs1}
    h_q(v_1) \ge w_2 \quad \textrm{ and } \quad
    h_q(v) \ge w_1\quad \forall v \in [v_1, v_2].
  \end{equation}
\end{prop}

\begin{proof}
  By Lemma~\ref{le:rs1_cn1} and Lemma~\ref{le:rs1_cn2}, we need
  to prove that condition~(\ref{eq:inv_cond_rs1}) is sufficient
  in order $\ddom$ be invariant for the Riemann solver  $\Rsol^q_2$.
  Thus we assume that condition~(\ref{eq:inv_cond_rs1}) holds.

  Since $\ddom$ is invariant for~(\ref{eq:aw-rascle2}), it is sufficient
  to prove that the left and the right traces at $x=0$ for
  $\Rsol^q_2$ belong to $\ddom$.
  So fix $(\rho^l, v^l)$ and $(\rho^r, v^r)$  in $\ddom$.
  If $\Rsol^q_2((\rho^l, v^l),(\rho^r, v^r))$ produces the classical solution,
  then we conclude.
  Assume therefore that $\Rsol^q_2((\rho^l, v^l),(\rho^r, v^r))$
  does not produce the classical solution and denote with
  $(\hat \rho, \hat v)$ and $(\check \rho_2, \check v_2)$ the left and
  right traces at $x=0$ for $\Rsol^q_1((\rho^l, v^l),(\rho^r, v^r))$.\\
  If $(\hat \rho, \hat v) \not \in \ddom$, then every point
  $(\rho, v)$ of the Lax curve of the first
  family through $(\rho^l, v^l)$ contained in $\ddom$ has the property
  that $\rho v \le q$ and so the Riemann solver gives the classical solution,
  since waves of the second family have strictly positive speed.
  This permits to prove that $(\hat \rho, \hat v) \in \ddom$.\\
  If $(\check \rho_2, \check v_2) \not \in \ddom$, then
  every point $(\rho, v)$ of the Lax curve of the second
  family through $(\rho^r, v^r)$ contained in $\ddom$ has the property
  that $\rho v \le q$ and so the Riemann solver gives the classical solution.
  In fact, if $v^l > v^r$, then a shock wave of the first family
  with strictly negative speed appears, if $v^l = v^r$, then
  no wave of the first family appears, whereas
  if $v^l < v^r$, then all the states $(\rho, v)$ of the rarefaction wave
  have flux $\rho v$ less than or equal to $q$.
  This permits to prove that $(\check \rho_2, \check v_2) \in \ddom$.

  The proof is thus completed.
\end{proof}

%
%
% Estimates on total variation
%
%
\section{Total variation estimates for $\Rsol^q_1$ and $\Rsol^q_2$}
\label{se:tv_estimates}

In this section we make a comparison between the two Riemann solvers
$\Rsol^q_1$ and $\Rsol^q_2$ in terms
of the changes in the total variation of various
quantities.

Fix $(\rho^l, v^l) , (\rho^r, v^r) \in \dom$. 
We denote with $\tilde \rho_1$ and $\tilde \rho_2$ respectively
the $\rho$-component of $\rs1$ and of $\rs2$.
Moreover we denote with $\tilde v_1$, $\tilde v_2$ respectively
the $v$-component of $\rs1$ and of $\rs2$.
Finally, we put $\tilde y_1 = \tilde \rho_1 (\tilde v_1 + p(\tilde \rho_1))$,
$\tilde y_2 = \tilde \rho_2 (\tilde v_2 + p(\tilde \rho_2))$,
%$\tilde w_1 = \tilde v_1$, $\tilde w_2 = \tilde v_2$,
$\tilde w_1 = \tilde v_1 + p(\tilde \rho_1)$,
$\tilde w_2 = \tilde v_2 + p(\tilde \rho_2)$.

In order to facilitate the reading of the following calculation,
we refer to Figures~\ref{fig:1min2} and~\ref{fig:2min1}.

\begin{figure}[t!]
  \centering
  \begin{psfrags}
    \psfrag{0}{$0$} \psfrag{rhom}{$\rho^m$}
    \psfrag{rho1}{$\check\rho_2$} \psfrag{rho2}{$\check\rho_1$} \psfrag{rhohat}{$\hat\rho$}
    \psfrag{q}{$q$} \psfrag{rho}{$\rho$}  \psfrag{rhov}{$\rho v$}
    \psfrag{L1}{$L_1$} \psfrag{L2}{$L_2$} 
    \includegraphics[height=50mm]{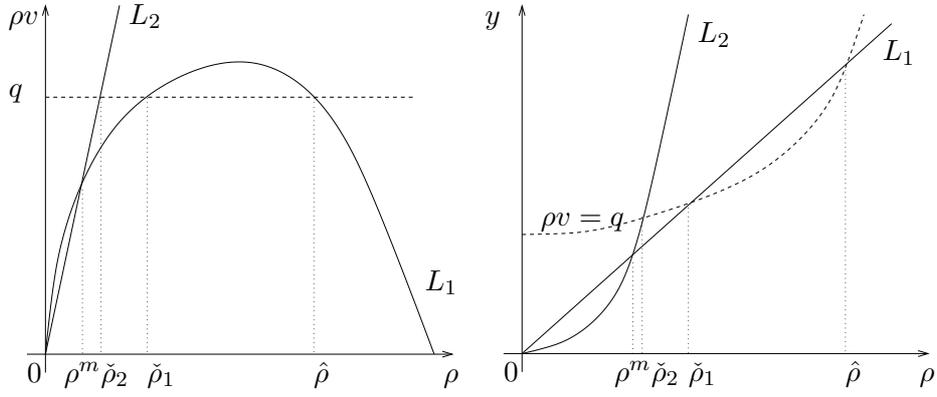}
  \end{psfrags}
  \hfil
  \begin{psfrags}
    \psfrag{0}{$0$} \psfrag{rhom}{$\rho^m$}
    \psfrag{rho1}{$\check\rho_2$} \psfrag{rho2}{$\check\rho_1$} \psfrag{rhohat}{$\hat\rho$}
    \psfrag{rv=qs}{$\rho v=q$} \psfrag{rho}{$\rho$}  \psfrag{y}{$y$}
    \psfrag{L2}{$L_1$} \psfrag{L1}{$L_2$} 
    \includegraphics[height=50mm]{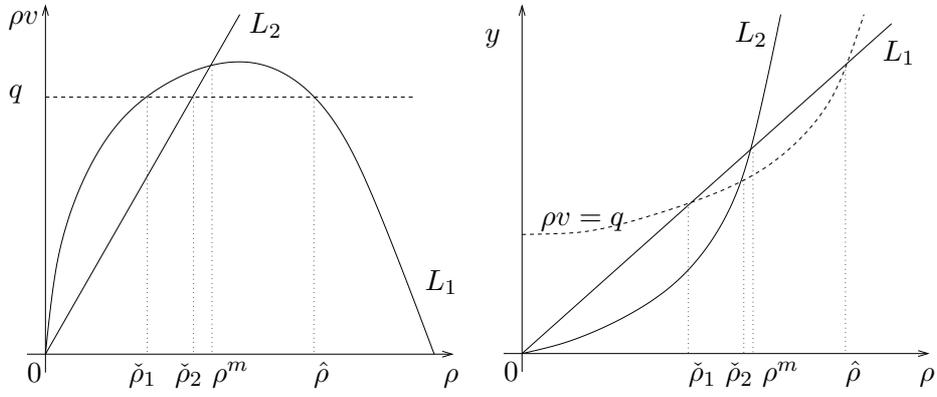}
  \end{psfrags}
  \caption{Notations used in the paper: case $\check\rho_1 > \check\rho_2$.}
  \label{fig:1min2}
\end{figure}

\vskip 10mm

\begin{figure}[h!]
  \centering
  \begin{psfrags}
    \psfrag{0}{$0$} \psfrag{rhom}{$\rho^m$}
    \psfrag{rho1}{$\check\rho_2$} \psfrag{rho2}{$\check\rho_1$} \psfrag{rhohat}{$\hat\rho$}
    \psfrag{q}{$q$} \psfrag{rho}{$\rho$}  \psfrag{rhov}{$\rho v$}
    \psfrag{L1}{$L_1$} \psfrag{L2}{$L_2$} 
    \includegraphics[height=50mm]{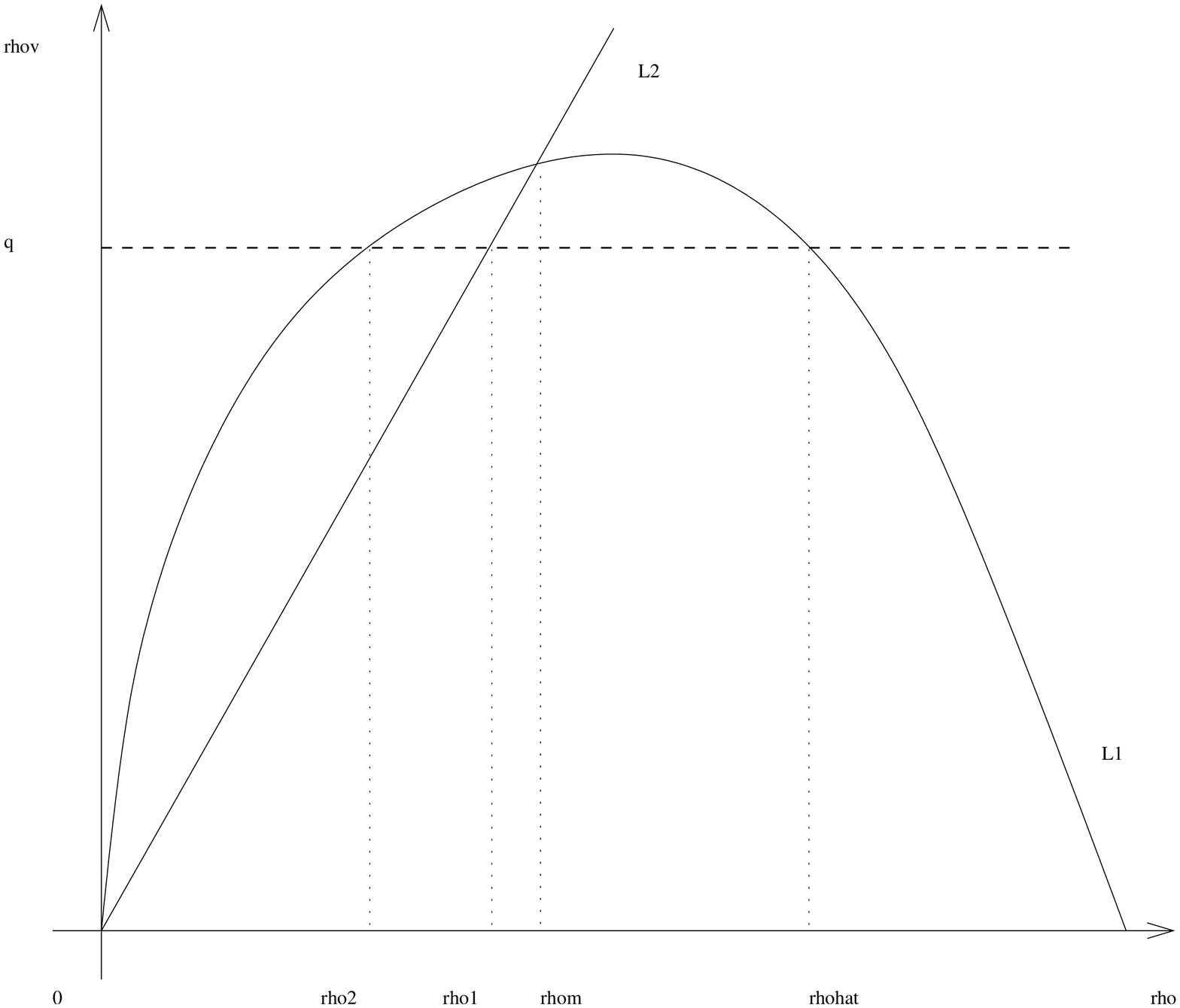}
  \end{psfrags}
  \hfil
  \begin{psfrags}
      \psfrag{0}{$0$} \psfrag{rhom}{$\rho^m$}
    \psfrag{rho1}{$\check\rho_2$} \psfrag{rho2}{$\check\rho_1$} \psfrag{rhohat}{$\hat\rho$}
    \psfrag{rv=qs}{$\rho v=q$} \psfrag{rho}{$\rho$}  \psfrag{y}{$y$}
    \psfrag{L2}{$L_1$} \psfrag{L1}{$L_2$} 
       \includegraphics[height=50mm]{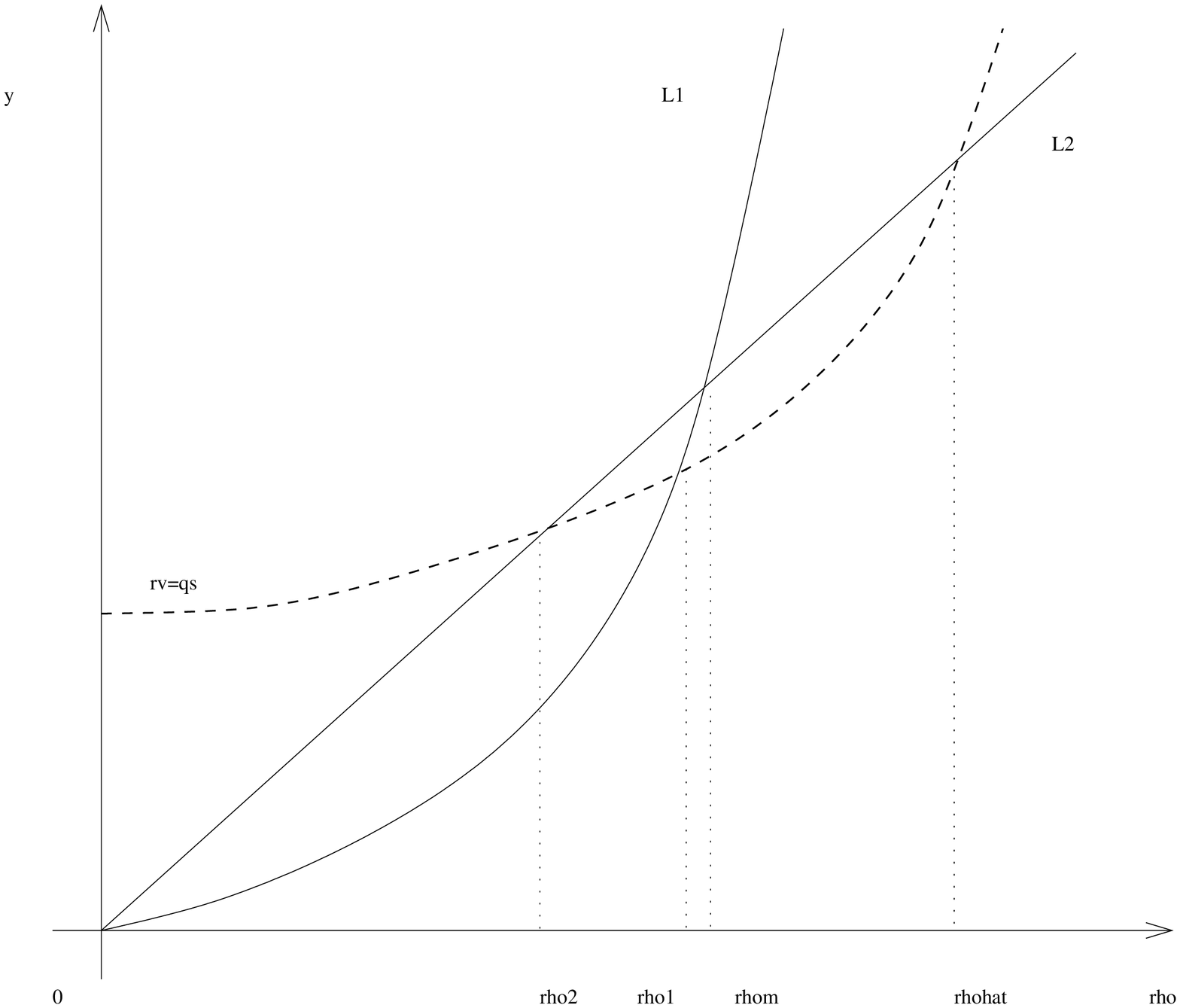}
  \end{psfrags}
  \caption{Notations used in the paper: case $\check\rho_1 < \check\rho_2$.}
  \label{fig:2min1}
\end{figure}

%
% tv density
%
\subsection{Total variation of the density $\rho$}
\label{sse:tv_density}

This subsection deals with $\tv (\tilde \rho_1)$ and $\tv (\tilde \rho_2)$.
The following proposition holds.

\begin{prop}
  For every initial conditions $(\rho^l, v^l), (\rho^r, v^r) \in \dom$, we have that
  \begin{equation}
    \label{eq:tv_rho}
    \tv (\tilde \rho_1) \geq \tv (\tilde \rho_2).
  \end{equation}
\end{prop}

\begin{proof}
  If $\rs1 = \rs2$, then $\tv(\tilde \rho_1) = \tv (\tilde \rho_2)$.
  Therefore, we assume  that 
  $$\rs1 \ne \rs2.$$
  In this case we have that
  $f_1(\Rsol ((\rho^l,v^l),(\rho^r,v^r))(0)) > q$
  and so, by construction of $\Rsol^1$ and
  $\Rsol^2$, we deduce that $\tilde \rho_1(x) = \tilde \rho_2(x)$
  for a.e. $x < 0$.
  Moreover, for $x>0$, $\tv \left( \tilde \rho_{2|_{]0, +\infty[}}\right)
  = \abs{\check \rho_2 - \rho^r}$, since the states
  $(\check \rho_2, \check v_2)$ and $(\rho^r, v^r)$ are connected
  by a contact discontinuity wave of the second family. Hence
  \begin{eqnarray*}
    \tv (\tilde \rho_2) & = & \tv \left( \tilde \rho_{2|_{]-\infty, 0[}}\right)
    + \abs{\hat \rho - \check \rho_2} + \abs{\check \rho_2 - \rho^r} \\
    & = & \abs{\rho^l - \hat \rho}
    + \abs{\hat \rho - \check \rho_2} + \abs{\check \rho_2 - \rho^r}.
  \end{eqnarray*}
  First consider the case $v^r = L_1(\rho^r; \rho^l, v^l)$,
  so that $(\check \rho_1, \check v_1)$ and $(\rho^r, v^r)$ can be
  connected by a wave of the first family. We get that
  \begin{eqnarray*}
    \tv (\tilde \rho_1) & = & \tv \left( \tilde \rho_{1|_{]-\infty, 0[}}\right)
    + \abs{\hat \rho - \check \rho_1} + \abs{\check \rho_1 - \rho^r} \\
    & = & \abs{\rho^l - \hat \rho}
    + \abs{\hat \rho - \check \rho_1} + \abs{\check \rho_1 - \rho^r}.
  \end{eqnarray*}
  If $\rho^r \leq \check \rho_1$, then 
  $\check \rho_2 \in ]\rho^r,\check \rho_1]$ and we obtain
  $\tv (\tilde \rho_1) = \tv (\tilde \rho_2)$. \\
  % = 
  %2 (\check \rho_1 - \rho^r) > 0$.\\ 
  %If $\rho^r = \check \rho_2$, then $\check \rho_1 = \check \rho_2$ and so
  %$\tv (\tilde \rho_1) = \tv (\tilde \rho_2)$. \\
  If $\rho^r > \check \rho_1$, then $\check \rho_1 < \check \rho_2 <
  \rho^r \le \hat \rho$ and so
  $\tv (\tilde \rho_1) - \tv (\tilde \rho_2) = 
  2 (\check \rho_2 - \check \rho_1) > 0$. \\
  Consider now the case $v^r \ne L_1(\rho^r; \rho^l, v^l)$.
  We have that
  \begin{displaymath}
    \tv (\tilde \rho_1) = \abs{\rho^l - \hat \rho} +
    \abs{\hat \rho - \check \rho_1} +
    \abs{\check \rho_1 - \rho^m} + \abs{\rho^m - \rho^r}.
  \end{displaymath}
  If $\check \rho_2 \leq \check \rho_1$, then we get that
  $\rho^m \leq \check \rho_2$ and 
  \begin{displaymath}
    \tv (\tilde \rho_1) - \tv (\tilde \rho_2) =  
    \check \rho_2 - \rho^m + \abs{\rho^m - \rho^r} - \abs{\check \rho_2 
      -\rho^r} \geq 0
  \end{displaymath}
  by the triangular inequality.
  If $\check \rho_2 > \check \rho_1$, then we get that
  $\rho^m > \check \rho_2$ and 
  \begin{eqnarray*}
    \tv (\tilde \rho_1) - \tv (\tilde \rho_2) & = & 
    \check \rho_2 + \rho^m - 2\check \rho_1 
    + \abs{\rho^m - \rho^r} - \abs{\check \rho_2 
      -\rho^r} \\
    & \ge & 2 (\check \rho_2 - \check \rho_1) > 0
  \end{eqnarray*}
  by the triangular inequality.
  This completes the proof.
\end{proof}

%
% tv velocity
%
\subsection{Total variation of the velocity $v$ (i.e. the first Riemann invariant)}
\label{sse:tv_velocity}

This subsection deals with the total variation of the velocity, i.e. of the
first Riemann invariant $z$.

\begin{prop}
  For every initial conditions $(\rho^l, v^l), (\rho^r, v^r) \in \dom$,
  we have that
  \begin{equation}
    \label{eq:tv_v}
    \tv (\tilde v_1) \geq \tv (\tilde v_2).
  \end{equation}
\end{prop}

\begin{proof}
  If $\rs1 = \rs2$, then the thesis clearly holds.
  Therefore we assume that 
  $$\rs1 \ne \rs2.$$ 
  In this situation we have that
  $f_1(\Rsol((\rho^l,v^l),(\rho^r,v^r))(0)) > q$ and so, by construction
  of $\rs1$ and $\rs2$, we deduce that $\tilde v_1(x) = \tilde v_2(x)$
  for a.e. $x<0$. It is clear that
  \begin{eqnarray*}
    \tv (\tilde v_2) & = & \abs{v^l - \hat v} + \abs{\hat v - \check v_2}
    + \abs{\check v_2 - v^r} \\
    & = & \abs{v^l - \hat v} + \abs{\hat v - v^r}
  \end{eqnarray*}
  since $\check v_2 = L_2(\check \rho_2; \rho^r, v^r) = v^r$.

  If $v^r = L_1(\rho^r; \rho^l, v^l)$, then
  %\begin{eqnarray*}
  $$
    \tv (\tilde v_1) = \abs{v^l - \hat v} + \abs{\hat v - \check v_1}
    + \abs{\check v_1 - v^r} 
    $$
%    & = & \abs{v^l - \hat v} + \abs{\hat v - \check v_2}
%    + \abs{\check v_2 - \check v_1}
  %\end{eqnarray*}
  and so, by the triangular inequality, we deduce
  $\tv (\tilde v_1) \geq \tv (\tilde v_2)$.

  If $v^r \ne L_1(\rho^r; \rho^l, v^l)$, then
  \begin{eqnarray*}
    \tv (\tilde v_1) & = & \abs{v^l - \hat v} + \abs{\hat v - \check v_1}
    + \abs{\check v_1 - v^m} + \abs{v^m - v^r} \\
    & = & \abs{v^l - \hat v} + \abs{\hat v - \check v_1}
    + \abs{\check v_1 - v^r}
  \end{eqnarray*}
  since $v^m = v^r$ by~(\ref{eq:rho2m}). 
  Again, $\tv (\tilde v_1) \geq \tv (\tilde v_2)$ by the triangular inequality.
  
%  Thus
%  \begin{eqnarray*}
%    \tv (\tilde v_1) & \le & \abs{v^l - \hat v} + \abs{\hat v - \check v_2}
%    + \abs{\check v_2 - v_2^m} + \abs{v_2^m - \check v_1} \\
%    & = & \tv (\tilde v_2) + \abs{v_2^m - \check v_1} = \tv (\tilde v_2)
%  \end{eqnarray*}
%  since $v_2^m = \check v_1$.

  The proof is so finished.
\end{proof}

%
% tv y
%
\subsection{Total variation of the generalized momentum $y$}
\label{sse:tv_y}

This subsection deals with the total variation of the
generalized momentum $y=\rho (v+p(\rho))$.

\begin{prop}
  Assume that hypothesis~(\ref{eq:pressure-function}) holds.
  For every initial conditions
  $(\rho^l, v^l), (\rho^r, v^r) \in \dom$, we have that
  \begin{equation}
    \label{eq:tv_y}
    \tv (\tilde y_1) \ge \tv (\tilde y_2).
  \end{equation}
\end{prop}

\begin{proof}
  If $\rs1 = \rs2$, then $\tv (\tilde y_1) = \tv (\tilde y_2)$.
  Therefore we assume that 
  $$\rs1 \ne \rs2.$$ 
  In this situation we have that
  $f_1(\Rsol ((\rho^l,v^l),(\rho^r,v^r))(0)) >q$
  and so $\tilde y_1(x) = \tilde y_2(x)$ for
  a.e. $x < 0$. We define $\hat y = \hat \rho (\hat v + p(\hat \rho))$,
  $y^l =  \rho^l ( v^l + p( \rho^l))$,
  $y^r =  \rho^r ( v^r + p( \rho^r))$,
  $\check y_1 = \check \rho_1 (\check v_1 + p(\check \rho_1))$,
  $\check y_2 = \check \rho_2 (\check v_2 + p(\check \rho_2))$.
  Note that $\hat y \ge \max \{ \check y_1, \check y_2 \}$.\\
  We have that
  \begin{displaymath}
    \tv(\tilde y_2) = \abs{y^l - \hat y} + \abs{\hat y - \check y_2}
    + \abs{\check y_2 - y^r}.
  \end{displaymath}
  Consider first the case $v^r = L_1(\rho^r; \rho^l, v^l)$,
  which implies that
  \begin{displaymath}
    \tv (\tilde y_1) = \abs{y^l - \hat y} + \abs{\hat y - \check y_1}
    + \abs{\check y_1 - y^r}.
  \end{displaymath}
  If $\rho^r \leq \check \rho_1$, then, by~(\ref{eq:pressure-function}),
  we easily get that
  $y^r \leq \check y_2 \leq \check y_1$ and consequently
  $\tv (\tilde y_1) = \tv (\tilde y_2)$.\\ 
  If $\rho^r > \check \rho_1$, then, by~(\ref{eq:pressure-function}),
  $y^r > \check y_2 > \check y_1$ and so
  $\tv (\tilde y_1) - \tv (\tilde y_2) = 2 (\check y_2 - \check y_1)$.

  Consider now the case $v^r \ne L_1(\rho^r; \rho^l, v^l)$.
  We have that
  \begin{displaymath}
    \tv (\tilde y_1) = \abs{y^l - \hat y} + \abs{\hat y - \check y_1}
    + \abs{\check y_1 - y^m} + \abs{y^m - y^r},
  \end{displaymath}
  where $y^m = \rho^m (v^m + p(\rho^m))$.\\
  If $\check \rho_1 \leq \check \rho_2$, then, by~(\ref{eq:pressure-function}),
  we deduce that $\check y_1 \leq \check y_2 \leq y^m$ and so
  \begin{eqnarray*}
    \tv (\tilde y_2) & = & \abs{y^l - \hat y} + (\hat y - \check y_2)
    + \abs{\check y_2 - y^r} \\
    & \leq & \abs{y^l - \hat y} + (\hat y - \check y_1)
    + \abs{\check y_2 - y^m} + \abs{y^m - y^r}\\
    & \leq & \abs{y^l - \hat y} + (\hat y - \check y_1)
    + \abs{\check y_1 - y^m} + \abs{y^m - y^r} = \tv( \tilde y_1).
  \end{eqnarray*}
  If $\check \rho_1 > \check \rho_2$, then, by~(\ref{eq:pressure-function}),
  we deduce that
  $\check y_1 > \check y_2 > y^m$ and so
  \begin{eqnarray*}
    \tv (\tilde y_2) & = & \abs{y^l - \hat y} + (\hat y - \check y_2)
    + \abs{\check y_2 - y^r} \\
    & \le & \abs{y^l - \hat y} + (\hat y - \check y_2)
    + \abs{\check y_2 - y^m} + \abs{y^m - y^r}\\
    & = & \abs{y^l - \hat y} + (\hat y - y^m)
    + \abs{y^m - y^r} = \tv( \tilde y_1).
  \end{eqnarray*}
  The proof is completed.
\end{proof}

%
% tv rc 2
%
\subsection{Total variation of the second Riemann invariant $w$}
\label{sse:tv_rc2}

This subsection deals with the total variation of the
second Riemann coordinate $w = v + p(\rho)$.

\begin{prop}
  For every initial conditions $(\rho^l, v^l), (\rho^r, v^r) \in \dom$, we have that
  \begin{equation}
    \label{eq:tv_z}
    \tv (\tilde w_1) \le \tv (\tilde w_2).
  \end{equation}
\end{prop}

\begin{proof}
  If $\rs1 = \rs2$, then $\tv (\tilde z_1) = \tv (\tilde z_2)$.
  Therefore we assume that 
  $$\rs1 \ne \rs2.$$ 
  In this situation we have that
  $f_1(\Rsol ((\rho^l,v^l),(\rho^r,v^r))(0)) >q$
  and so $\tilde w_1(x) = \tilde w_2(x)$ for
  a.e. $x < 0$. We define $\hat w = \hat v + p(\hat \rho)$,
  $w^l =  v^l + p( \rho^l)$,
  $w^r =  v^r + p( \rho^r)$,
  $\check w_1 = \check v_1 + p(\check \rho_1)$,
  $\check w_2 = \check v_2 + p(\check \rho_2)$.
  Note that $w^l=\hat w = \check w_1$.\\
  We have that
  \begin{displaymath}
    \tv(\tilde w_2) = \abs{w^l - \hat w} + \abs{\hat w - \check w_2}
    + \abs{\check w_2 - w^r}.
  \end{displaymath}
  Consider first the case $v^r = L_1(\rho^r; \rho^l, v^l)$,
  which implies that
  \begin{displaymath}
    \tv (\tilde w_1) = \abs{w^l - \hat w} \le \tv (\tilde w_2).
  \end{displaymath}
  Consider now the case $v^r \ne L_1(\rho^r; \rho^l, v^l)$.
  In this case we have that
  \begin{displaymath}
    \tv (\tilde w_1) = \abs{w^l - \hat w} + \abs{w^m - w^r},
  \end{displaymath}
  where $w^m = v^m + p(\rho^m)$. Since $w^m = \hat w$, we
  conclude by the triangular inequality.

  The proof is so finished.
\end{proof}

%%%%%%%%%%%%%%%%%%%%%%%%%%%%%%%%%%%%%%%%%%%%%

\section{Numerical schemes}\label{se:numerics}

This section is devoted to the construction of finite volume numerical schemes 
to capture the solutions corresponding to $\Rsol_{1}^q$ and $\Rsol_{2}^q$.

Let $\Delta x$ and $\Delta t$ be two constant increments
for space and time discretization.
We then define the mesh interfaces
$x_{j+1/2} = j \Delta x$ (so that $x_{1/2}=0$ corresponds to
the constraint location) and the cell centers
$x_{j} = (j - 1/2) \Delta x$ for $j \in \mathbb{Z}$,
the intermediate times $t^n = n \Delta t$ for
$n \in \mathbb{N}$, and at each time $t^n$
we denote ${\bf u}^n_j$ an
approximate mean value of the solution of
\eqref{eq:aw-rascle2}, \eqref{eq:constraint} on the interval
$\mathcal{C}_{j} = [x_{j-1/2}, x_{j+1/2})$, $j \in \mathbb{Z}$.
In other words, a piecewise constant approximation
$x \to {\bf u}(t^n, x)$ of the conserved variables
${\bf u} = (\rho, y)$ is given by
$$
{\bf u}(t^n, x) = {\bf u}^n_{j} \,\,\, {\mbox{for all}} \,\,\,
x \in \mathcal{C}_{j}, \,\,\, j \in \mathbb{Z}, \,\,\, n \in
\mathbb{N}.
$$
When $n=0$, we set
\begin{equation} \label{initialisation}
  {\bf u}^0_{j}= \frac{1}{\Delta x} \int_{x_{j-1/2}}^{x_{j+1/2}}
  {\bf u}_0(x) dx, \,\,\, {\mbox{for all}} \,\,\,j \in \mathbb{Z},
\end{equation}
where ${\bf u}_0= (\rho_0, y_0)\in\dom$ is a given initial data
(we will restrict the study to Riemann-type initial data).

Given a sequence $({\bf u}^n_j)_{j \in \mathbb{Z}}$ at time $t^n$,
we concentrate now on the computation of an approximate solution
at the next time level $t^{n+1}$.

We will concentrate on Godunov scheme and show how to adapt it
in order to match the constraint condition~\eqref{eq:constraint} at $x=0$.   
We recall that, as pointed out in~\cite{MR2352330}, classical conservative
schemes (like Godunov method) may generate important 
non-physical oscillations near contact discontinuities. For this reason
we will restrict to Riemann data lying on the same second Riemann invariant,
i.e. we take $v^l+p(\rho^l)=v^r+p(\rho^r)$. More general cases can be treated for example combining
the techniques presented here with the Transport-Equilibrium
scheme described in~\cite{MR2352330}. Note that, in any case, a contact discontinuity
appears when applying the Riemann solver $\Rsol_{2}^q$.

For sake of completeness, we recall that classical Godunov scheme writes
\begin{equation}
  \label{unp1j}
  \bu^{n+1}_j = \bu^{n}_j - \frac{\Delta t}{\Delta x}
  ({\bf f}^{n}_{j+1/2} - {\bf f}^{n}_{j-1/2})
  \,\,\, \mbox{for all} \,\,\, j \, \in \, \mathbb{Z},
\end{equation}
where the numerical fluxes are given by
\begin{equation}
  \label{flux_god}
  {\bf f}^{n}_{j+1/2} = {\bf f}(\bu^{n}_j,\bu^{n}_{j+1}) =
  {f}(\Rsol (\bu^{n}_j,\bu^{n}_{j+1}) (0))
  \,\,\, \mbox{for all} \,\,\, j \, \in \, \mathbb{Z},
\end{equation}
and the usual CFL condition
\begin{equation}
  \label{cfl}
  \frac{\Delta t}{\Delta x} \max_{j  \in  \mathbb{Z}}\{|\lambda_i(\bu^n_j)|,\,\,
  i=1,2
  \} \leq \frac{1}{2}
\end{equation}
holds.
%for all the $\bu$ under consideration.
In the following sections we describe how to modify the definition of
the numerical flux~\eqref{flux_god}
for $j=0$. The simulations have been performed taking
$p(\rho)=\rho$ and $\Delta x=0.002$.

%%%%%
\subsection{The Constrained Godunov scheme for $\Rsol_{1}^q$}
\label{sse:CG1}

We follow the idea introduced in~\cite{AGS} for the scalar case. 
We redefine the numerical flux at the interface $x_{1/2}=0$ to take into account the 
imposed constraint~\eqref{eq:constraint}. 
We denote by ${\bf f}^{n}_{1,j+1/2}$, ${\bf f}^{n}_{2,j+1/2}$ the components of the 
classical Godunov flux:
$$
{\bf f}^{n}_{j+1/2} =
\left(
\begin{array}{c}
{\bf f}^{n}_{1,j+1/2} \\
{\bf f}^{n}_{2,j+1/2}
\end{array}
\right) \,.
$$
For $j=0$, we replace it by $\hat{\bf f}^{n}_{1/2}$, where
\be \label{eq:RS1flux}
\begin{array}{rcl}
\hat{\bf f}^{n}_{1,1/2} &=& \min\left\{ {\bf f}^{n}_{1,1/2} \ , \ q  \right\}\,, \\
\hat{\bf f}^{n}_{2,1/2} &=& \hat{\bf f}^{n}_{1,1/2} \ \displaystyle{\frac{ {\bf f}^{n}_{2,1/2} }{{\bf f}^{n}_{1,1/2}}}
=  \min\left\{  {\bf f}^{n}_{2,1/2} \ , \ q \ \displaystyle{\frac{ {\bf f}^{n}_{2,1/2} }{{\bf f}^{n}_{1,1/2}}}   \right\}  \,.
\end{array}
\ee
We stress that the above construction preserves conservation, 
in agreement with the conservative character of $\Rsol_{1}^q$.

\begin{theorem} \label{thm:max}
  {\rm\bf (Maximum principle)}
  Under the CFL restriction \eqref{cfl}, the finite volume
  numerical scheme defined by
  \eqref{unp1j}, \eqref{flux_god} and \eqref{eq:RS1flux}
  satisfies the maximum principle property
  \begin{displaymath}
    \inf_{l\in\mathbb{Z}} \left( v_l^0 + p(\rho_l^0 )\right)
    \le v_j^n + p(\rho_j^n ) \le
    \max_{l\in\mathbb{Z}} \left( v_l^0 + p(\rho_l^0 )\right)
  \end{displaymath}
  for all $j\in\mathbb{Z}$ and all $n\in\N$, where
  ${\bf u}_j^n = \left(
    \begin{array}{c}
      \rho_j^n \\
      y_j^n
    \end{array}
  \right)$ and $v_j^n + p(\rho_j^n) = \frac{y_j^n}{\rho_j^n} = w_j^n$.
\end{theorem}

\begin{proof}
  We observe first that the above maximum principle property on the
  second Riemann invariant is
  satisfied by the classical Godunov scheme \eqref{unp1j}, \eqref{flux_god} 
  (see for example \cite[Remark 3.1 (ii)]{MR2352330} for
  a detailed computation).
  Thus we only need to check what happens for $j=0,1$.

  If $\hat{\bf f}^{n}_{1,1/2} =  {\bf f}^{n}_{1,1/2}$, then also
  $\hat{\bf f}^{n}_{2,1/2} =  {\bf f}^{n}_{2,1/2}$
  and the scheme reduces to the classical Godunov scheme.
  Therefore we assume that
  $\hat{\bf f}^{n}_{1,1/2} = q < {\bf f}^{n}_{1,1/2}$ and 
  $\hat{\bf f}^{n}_{2,1/2} = q \, {\bf f}^{n}_{2,1/2} / {\bf f}^{n}_{1,1/2} 
  <  {\bf f}^{n}_{2,1/2}$.
  In this case, recalling the construction of $\Rsol_{1}^q$,
  it is easy to see that 
  \begin{displaymath}
    \begin{array}{rcl}
      \hat{\bf f}^{n}_{1,1/2} &=&  {\bf f}_1 (\bu_0^n, \hat\bu) =
      {\bf f}_1 (\check\bu_1, \bu_1^n) \,, \\
      \hat{\bf f}^{n}_{2,1/2} &=& q\, w(\Rsol\left(\bu_0^n, \bu_1^n \right)(0) )
      =  {\bf f}_2 (\bu_0^n, \hat\bu) =  {\bf f}_2 (\check\bu_1, \bu_1^n)  \,,
    \end{array}
  \end{displaymath}
  where $\hat\bu$ and $\check\bu_1$ are, respectively, the left
  and right traces at $x=0$ of  $\Rsol_{1}^q (\bu_0^n,\bu_1^n)$.
  In fact, since $\Rsol\left(\bu_0^n,  \hat\bu \right)$ counts
  only waves of negative speed, we have that
   \begin{displaymath}
     {\bf f} (\bu_0^n, \hat\bu) = f\left( \Rsol\left(\bu_0^n, 
         \hat\bu \right) (0) \right) = f (\hat\bu)
     =
     \left(
       \begin{array}{c}
         q \\
         q w_0^n
       \end{array}
     \right)
     \,.
   \end{displaymath}
   On the other side, since $\Rsol\left(\check\bu_1, \bu_1^n \right)$
   counts only waves of positive speed, we have that
   \begin{displaymath}
     {\bf f} (\check\bu_1, \bu_1^n) = 
     f\left( \Rsol\left(\check\bu_1, \bu_1^n \right) (0) \right) 
     = f (\check\bu_1)
     =
     \left(
       \begin{array}{c}
         q \\
         q w_0^n
       \end{array}
     \right)
     \,.
   \end{displaymath}
   Hence the following bounds hold for $j=0,1$:
   \begin{displaymath}
     \inf_{l=-1,\ldots,2} \left\{ w_l^n , \hat w, \check w_1 \right\}
     \le w_j^n  \le
     \max_{l=-1,\ldots,2} \left\{ w_l^n , \hat w, \check w_1 \right\} ,
   \end{displaymath}
   and we conclude observing that $\hat w = \check w_1 = w_0^n$.
\end{proof}

We have tested our method on Riemann data lying on the same 
1-Riemann invariant, in order to avoid spurious oscillations due to 
the presence of contact discontinuities. More general data can be dealt with
using the technique presented in~\cite{MR2352330}.
Figures~\ref{rhov-test1a},~\ref{rhov-test1b}, 
shows that the numerical solutions are in good agreement with 
exact solutions. In particular, our scheme perfectly
captures the nonclassical shock at $x=0$.

\begin{figure}[htbp]
  \begin{center}
    \includegraphics[width=6.2cm]{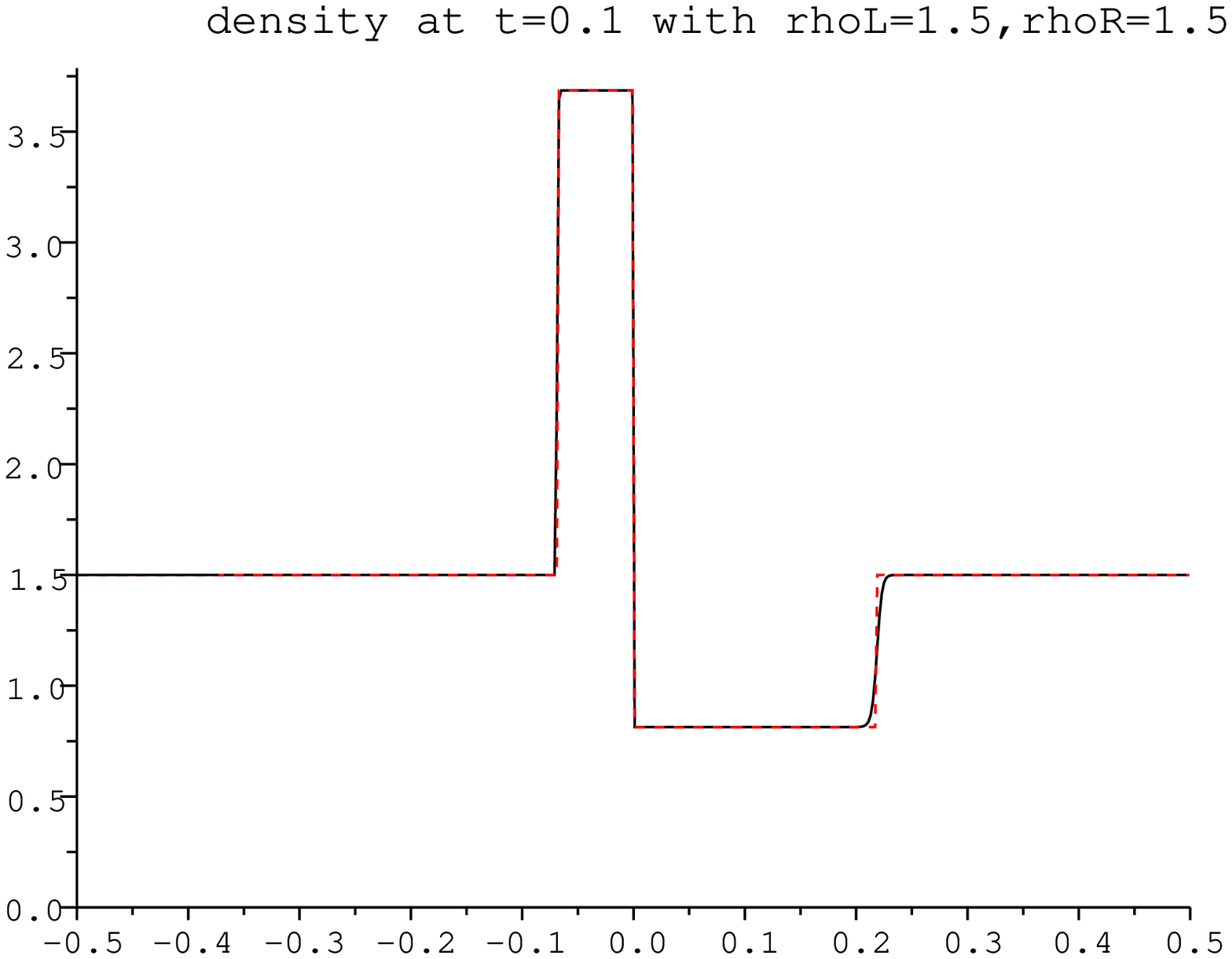} 
    \includegraphics[width=6.2cm]{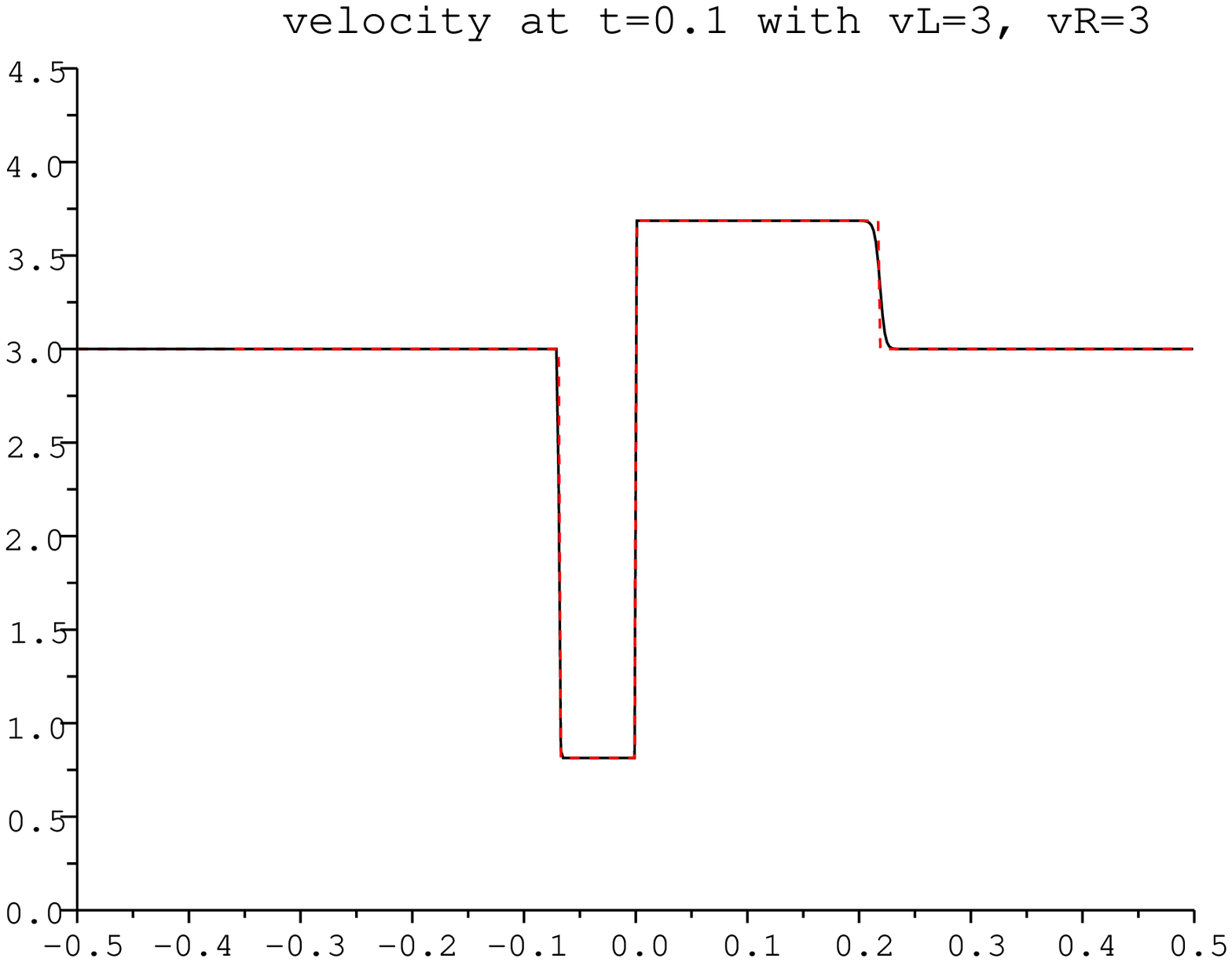}
    \caption{{\bf Test 1a} : Solution of the constrained
      Riemann solver $\Rsol_{1}^q$
      with data $\rho^l=\rho^r=1.5$, 
      $v^l=v^r=3$ and $q=3$: exact solution (dashed line),
      numerical approximation (continuous line).}
    \label{rhov-test1a}
  \end{center}
\end{figure}

\begin{figure}[htbp]
  \begin{center}
    \includegraphics[width=6.2cm]{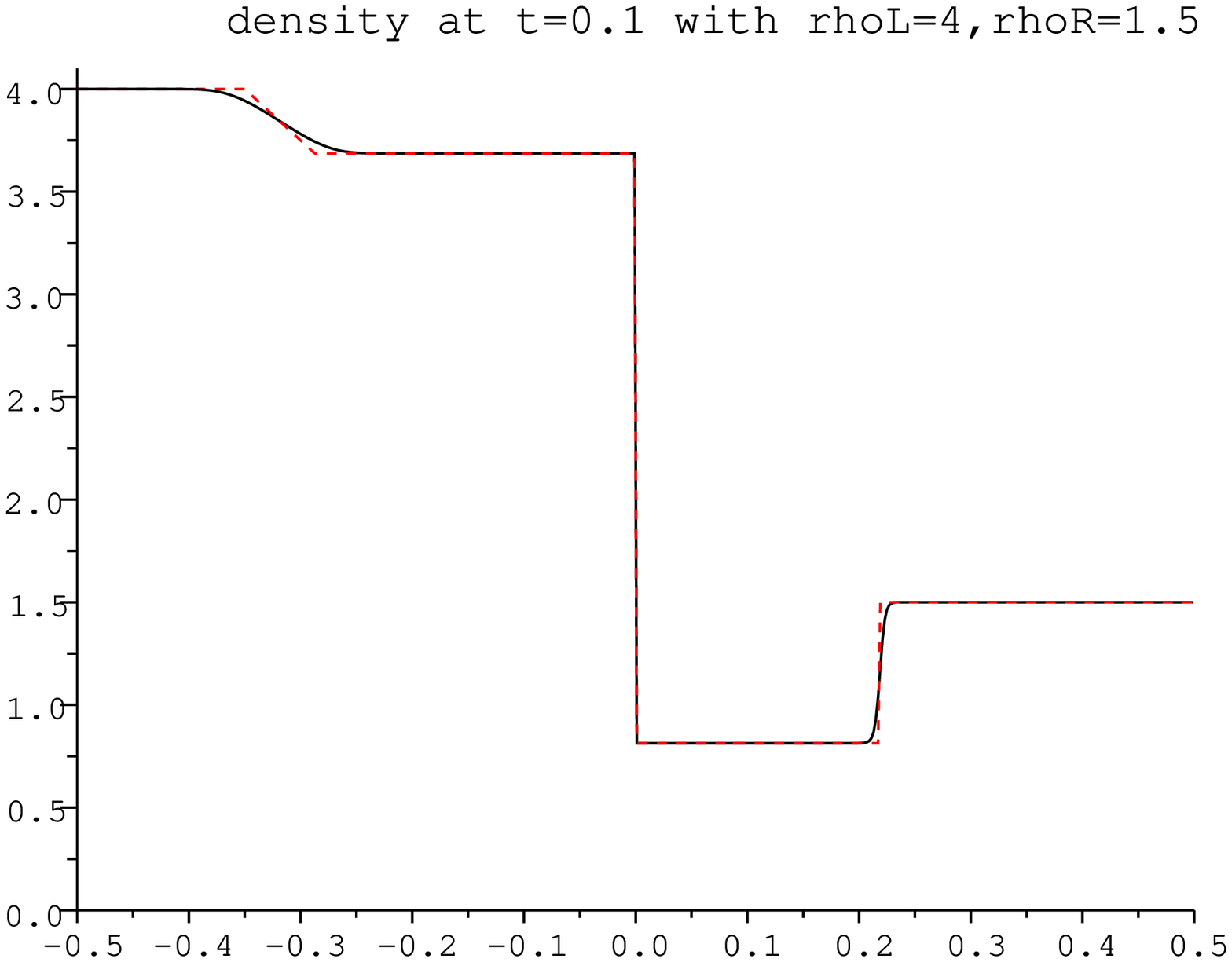} 
    \includegraphics[width=6.2cm]{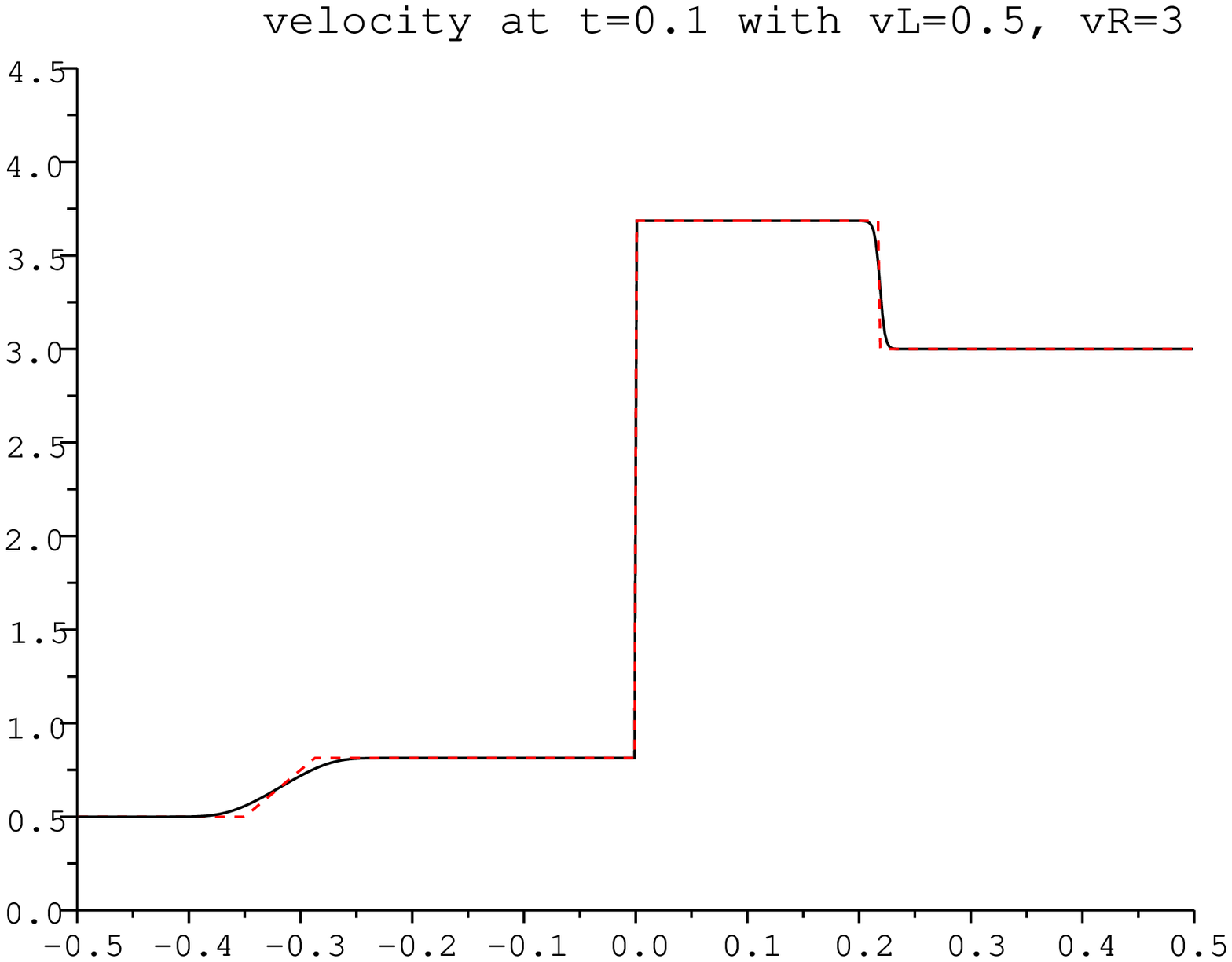}
    \caption{{\bf Test 1b} : Solution of the constrained Riemann
      solver $\Rsol_{1}^q$
      with data $\rho^l=4$, $\rho^r=1.5$,
      $v^l=0.5$, $v^r=3$ and $q=3$: exact solution (dashed line),
      numerical approximation (continuous line).}
    \label{rhov-test1b}
  \end{center}
\end{figure}

%%%%%
\subsection{The Constrained Godunov scheme for $\Rsol_{2}^q$}
\label{sse:CG2}

The Constrained Riemann Solver $\Rsol_{2}^q$ is not globally
conservative at the point $x=0$
(by definition, conservation holds only for the first
equation in~\eqref{eq:aw-rascle2} and therefore only car
density $\rho$ is conserved).
As a consequence, we look for a non-conservative
numerical scheme, i.e. we define two numerical fluxes
$\tilde{\bf f}^{n,-}_{1/2} \not= \tilde{\bf f}^{n,+}_{1/2}$ such that
\begin{eqnarray} 
  \bu^{n+1}_0 &=& \bu^{n}_0 - \frac{\Delta t}{\Delta x}
  (\tilde{\bf f}^{n,-}_{1/2} - {\bf f}^{n}_{-1/2}) \,, \label{eq:tracciasx} \\
  \bu^{n+1}_1 &=& \bu^{n}_1 - \frac{\Delta t}{\Delta x}
  ({\bf f}^{n}_{3/2} - \tilde{\bf f}^{n,+}_{1/2}) \label{eq:tracciadx} \,.
\end{eqnarray}
We set
\begin{displaymath}
  \tilde{\bf f}^{n,-}_{1,1/2} = \tilde{\bf f}^{n,+}_{1,1/2} 
  = \min\left\{ {\bf f}^{n}_{1,1/2}, q  \right\}\,,
\end{displaymath}
and 
$$
\tilde{\bf f}^{n,-}_{2,1/2} = 
\tilde{\bf f}^{n}_{1,1/2} \ \displaystyle{\frac{ {\bf f}^{n}_{2,1/2} }{{\bf f}^{n}_{1,1/2}}}
$$
In order to capture the right trace at $x=0$, we could envisage using a \emph{ghost cell}
type method (introduced in~\cite{MR1699710}, see also~\cite{MR2377108} and references therein for other applications), computing the ghost value $\check\bu_1^n$ corresponding to 
$\bu_1^n$, whose $(\rho,v)$ components are given by
$$
\check\rho_1^n = q/v_1^n \,, \quad
\check v_1^n = v_1^n \,,
$$
where $v^n_1=\displaystyle{\frac{y^n_1}{\rho^n_1}-p(\rho^n_1)}$.
This is obtained using the following flux
$$
  \tilde{\bf f}^{n,+}_{2,1/2} = q\left( v^n_1 +p(q/v^n_1)\right) \,,
$$
whenever ${\bf f}^{n}_{1,1/2} < q$.
Unfortunately, due to the convexity assumption~\eqref{eq:pressure-function} on the function $\rho\mapsto\rho p(\rho)$,
the velocity component is overestimated during the projection step of Godunov scheme in~\eqref{eq:tracciadx}
(see~\cite{MR2352330}). Therefore,  the right trace cannot be captured properly: the velocity component is overestimated 
and the density is underestimated,
see Figures~\ref{rhov-test2a}, \ref{rhov-test2b}.
In fact, at each time-step, we have $\check v_1^{n+1} = v_1^{n+1} \ge v_1^{n}$ and $\check\rho_1^{n+1} \le \check\rho_1^{n}$,
where the inequality is strict generally speaking.

In order to overcome this difficulty, we propose to simply keep the value of the velocity component fixed for $j=1$,
i.e. to replace the value obtained by~\eqref{eq:tracciadx} with $\tilde{\bf f}^{n,+}_{2,1/2} = \tilde{\bf f}^{n,-}_{2,1/2}$
by $v_1^{n+1} = v_1^{n}$, and then updating the conservative component as 
$$
y_1^{n+1} = \rho_1^{n+1} \left( v_1^{n} + p(\rho_1^{n+1}) \right) \,,
$$
whenever ${\bf f}^{n}_{1,1/2} < q$.
This allows to capture precisely the right trace of the discontinuity at $x=0$, as shown by numerical simulation
in Figures~\ref{rhov-test2a}, \ref{rhov-test2b}. Only, a small amplitude oscillation
traveling at speed $v=v^r$ is produced, see Figures~\ref{rhov-test2a-zoom}, \ref{rhov-test2b-zoom}.

\begin{figure}[htbp]
  \begin{center}
    \includegraphics[width=6.2cm]{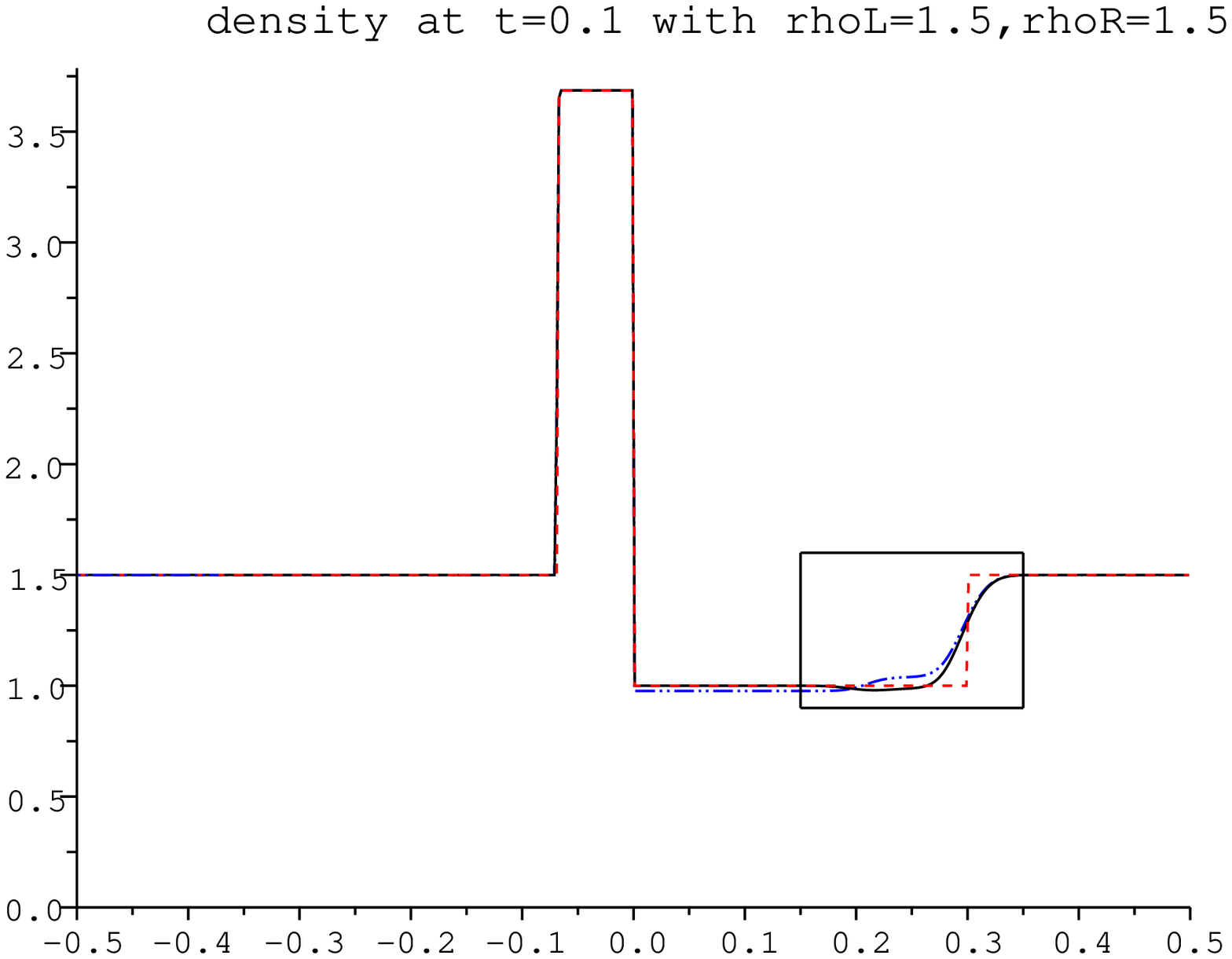} 
    \includegraphics[width=6.2cm]{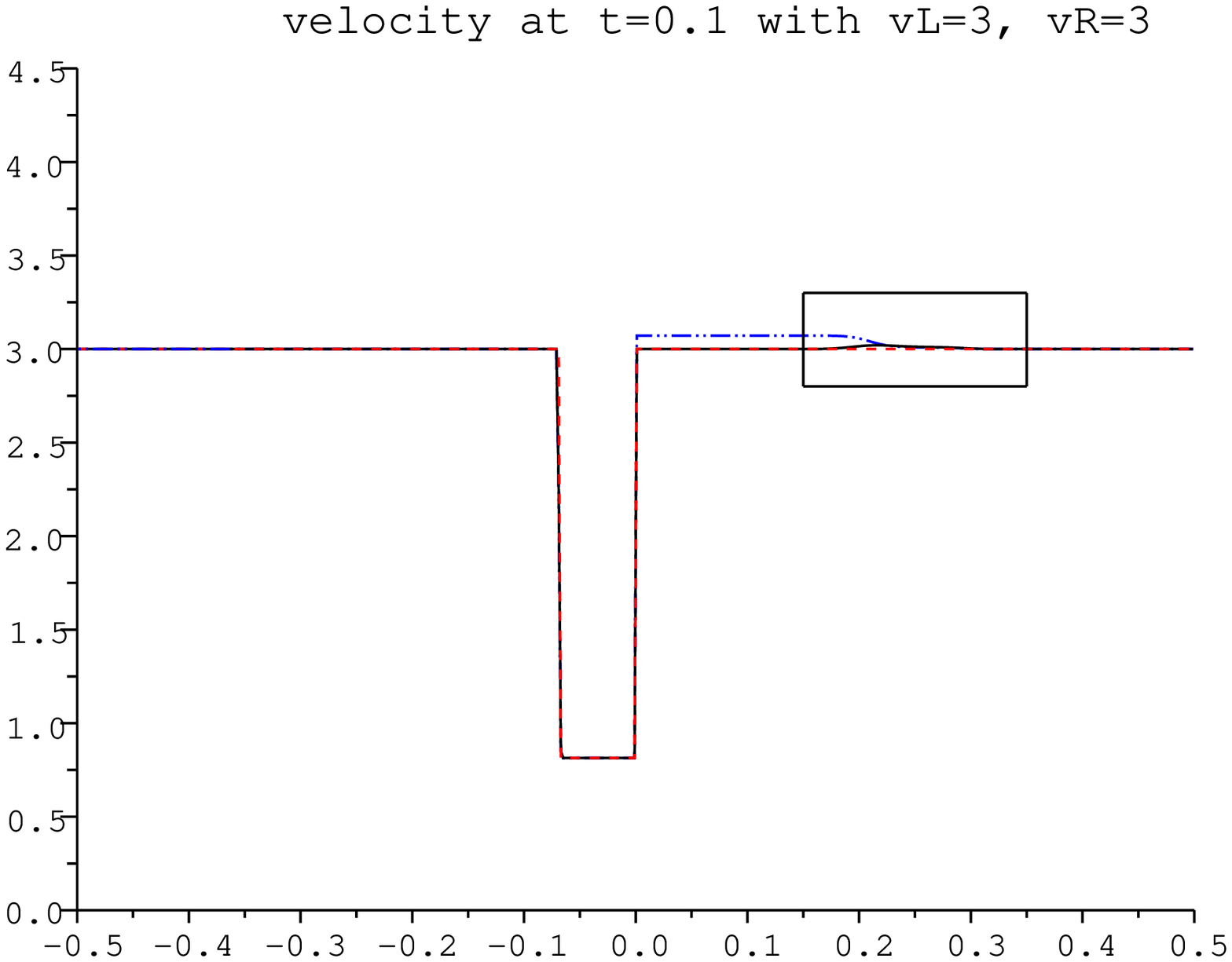}
    \caption{{\bf Test 2a}: Solution of the constrained Riemann
      solver $\Rsol_{2}^q$ with
      data $\rho^l=\rho^r=1.5$, $v^l=v^r=3$ and $q=3$:
      exact solution (dashed line), ghost cell method  (dash-dotted line),
      our method (continuous line).
      The rectangles select the zoomed areas plotted
      in Figure~\ref{rhov-test2a-zoom}.}
    \label{rhov-test2a}
  \end{center}
\end{figure}

\begin{figure}[htbp]
  \begin{center}
    \includegraphics[width=6.2cm]{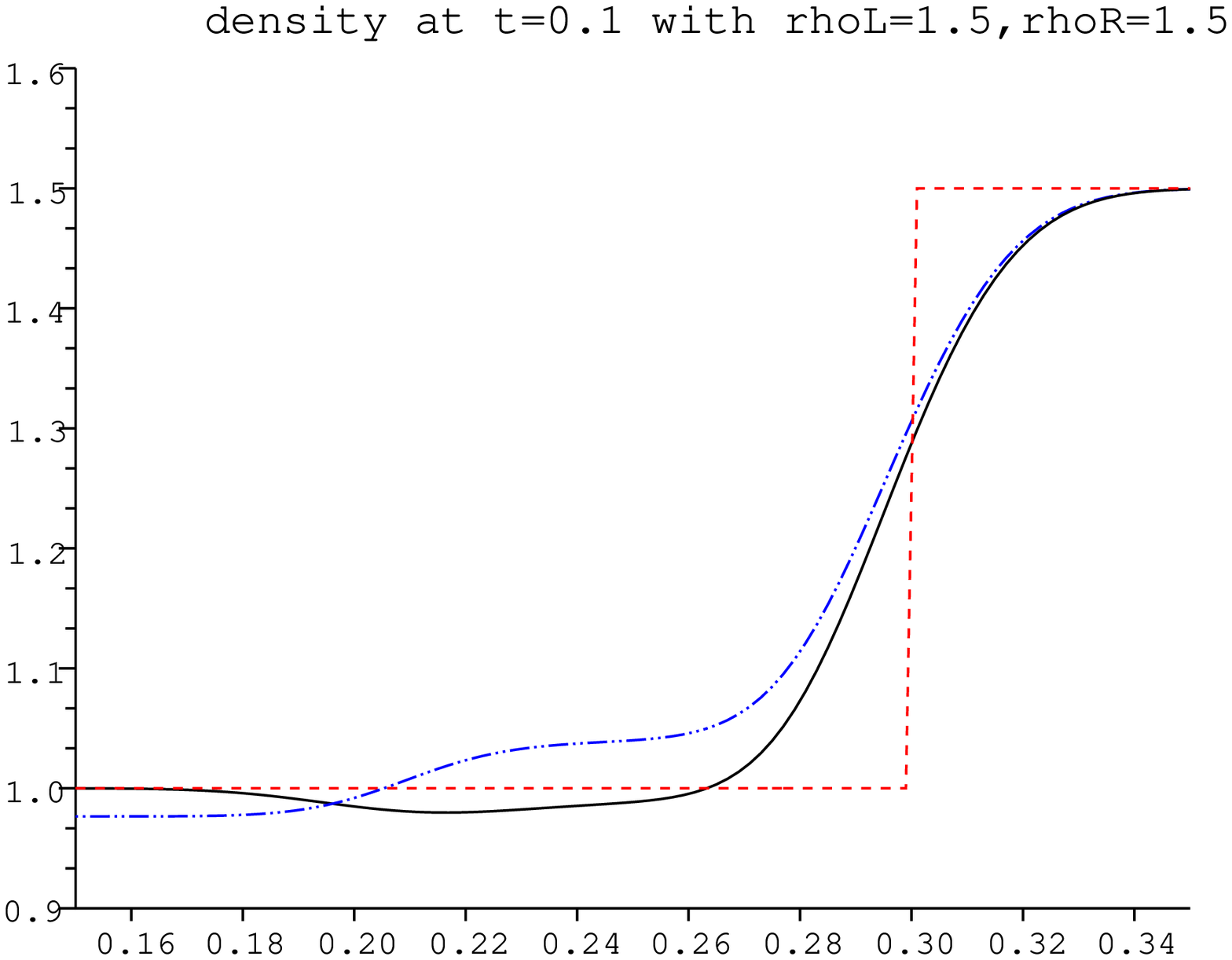} 
    \includegraphics[width=6.2cm]{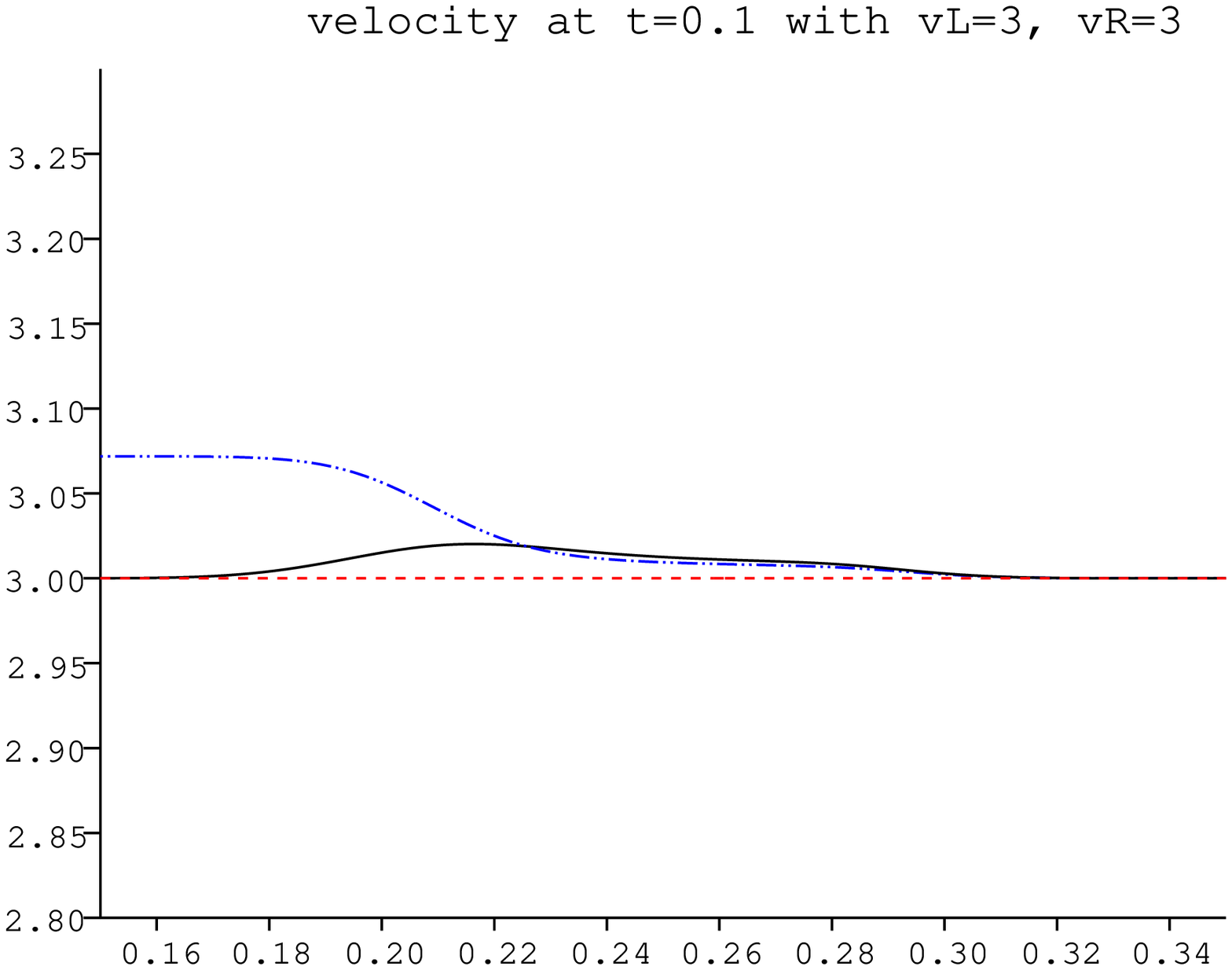}
    \caption{{\bf Test 2a}: Detailed view of a part of the
      computational domain of  Figure \ref{rhov-test2a}.
%    Solution of the constrained Riemann solver $\Rsol_{2}^q$ with
%      data $\rho^l=\rho^r=1.5$, $v^l=v^r=3$ and $q=3$: exact solution (dashed line), ghost cell method  (dash-dotted line),
%      our method (continuous line).
}
    \label{rhov-test2a-zoom}
  \end{center}
\end{figure}

\begin{figure}[htbp]
  \begin{center}
    \includegraphics[width=6.2cm]{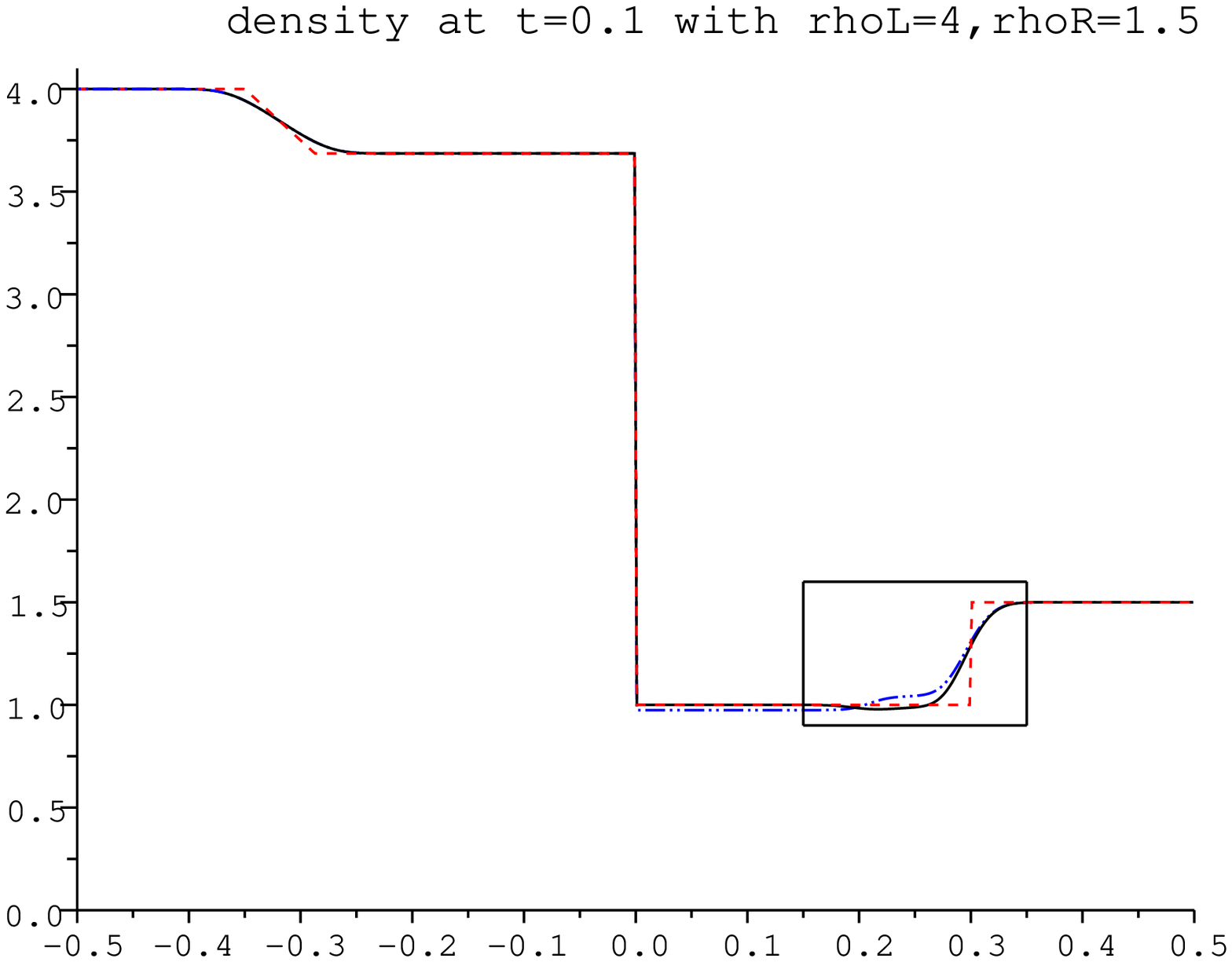} 
    \includegraphics[width=6.2cm]{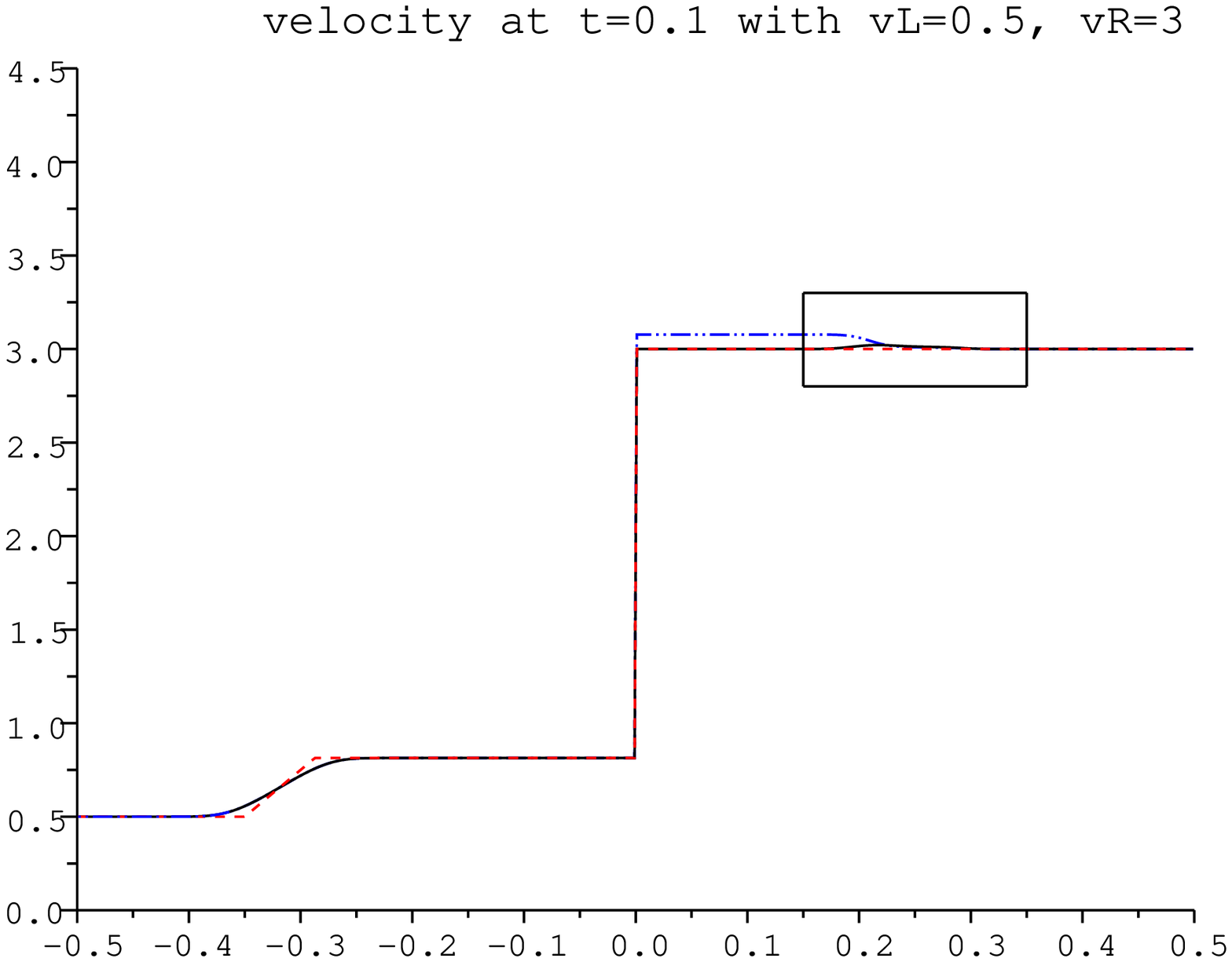}
    \caption{{\bf Test 2b}: Solution of the constrained
      Riemann solver $\Rsol_{2}^q$
      with data $\rho^l=4$, $\rho^r=1.5$,
      $v^l=0.5$, $v^r=3$ and $q=3$: exact solution (dashed line),
      ghost cell method  (dash-dotted line),
      our method (continuous line). The rectangles select the
      zoomed areas plotted in Figure~\ref{rhov-test2b-zoom}.}
    \label{rhov-test2b}
  \end{center}
\end{figure}

\begin{figure}[htbp]
  \begin{center}
    \includegraphics[width=6.2cm]{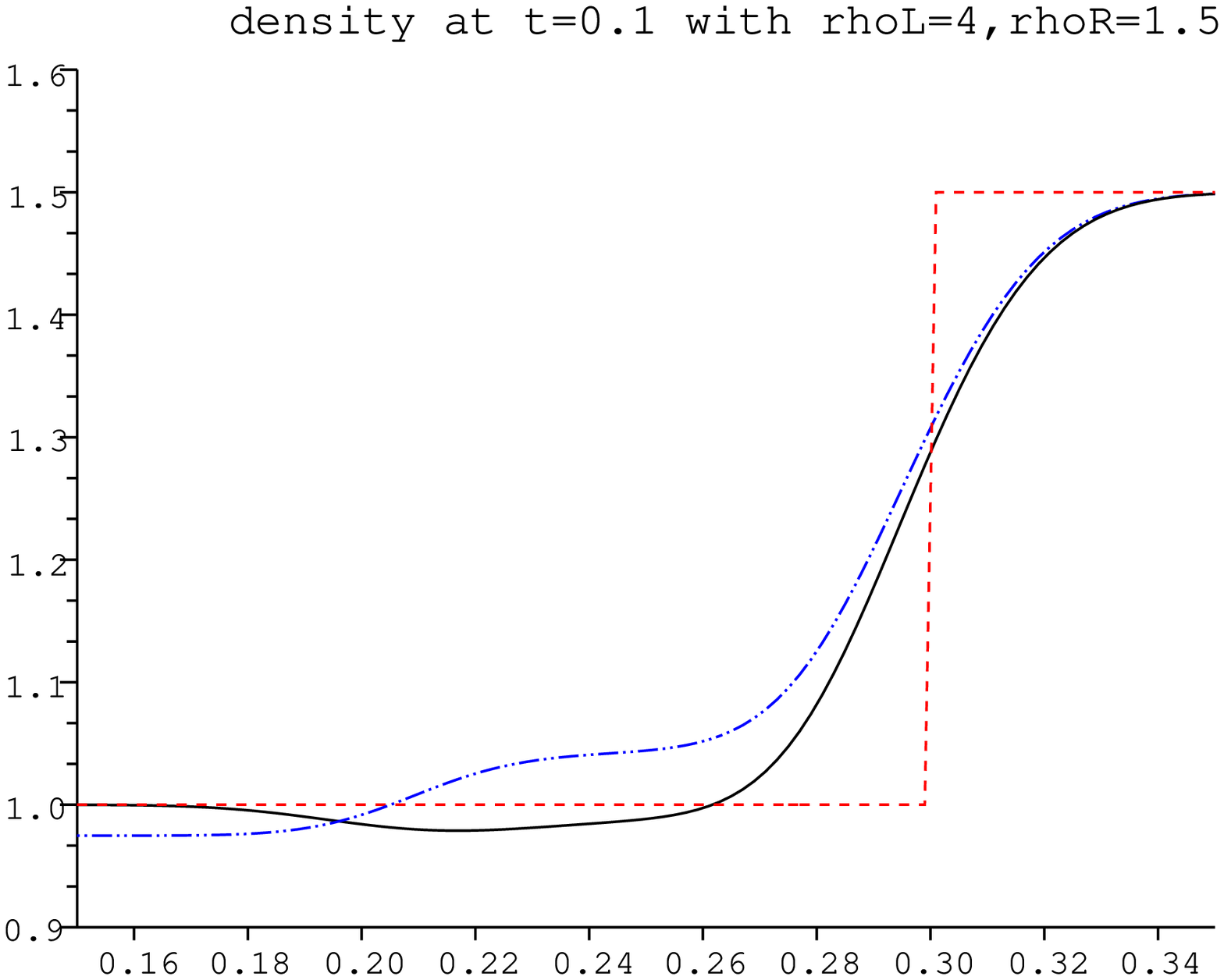} 
    \includegraphics[width=6.2cm]{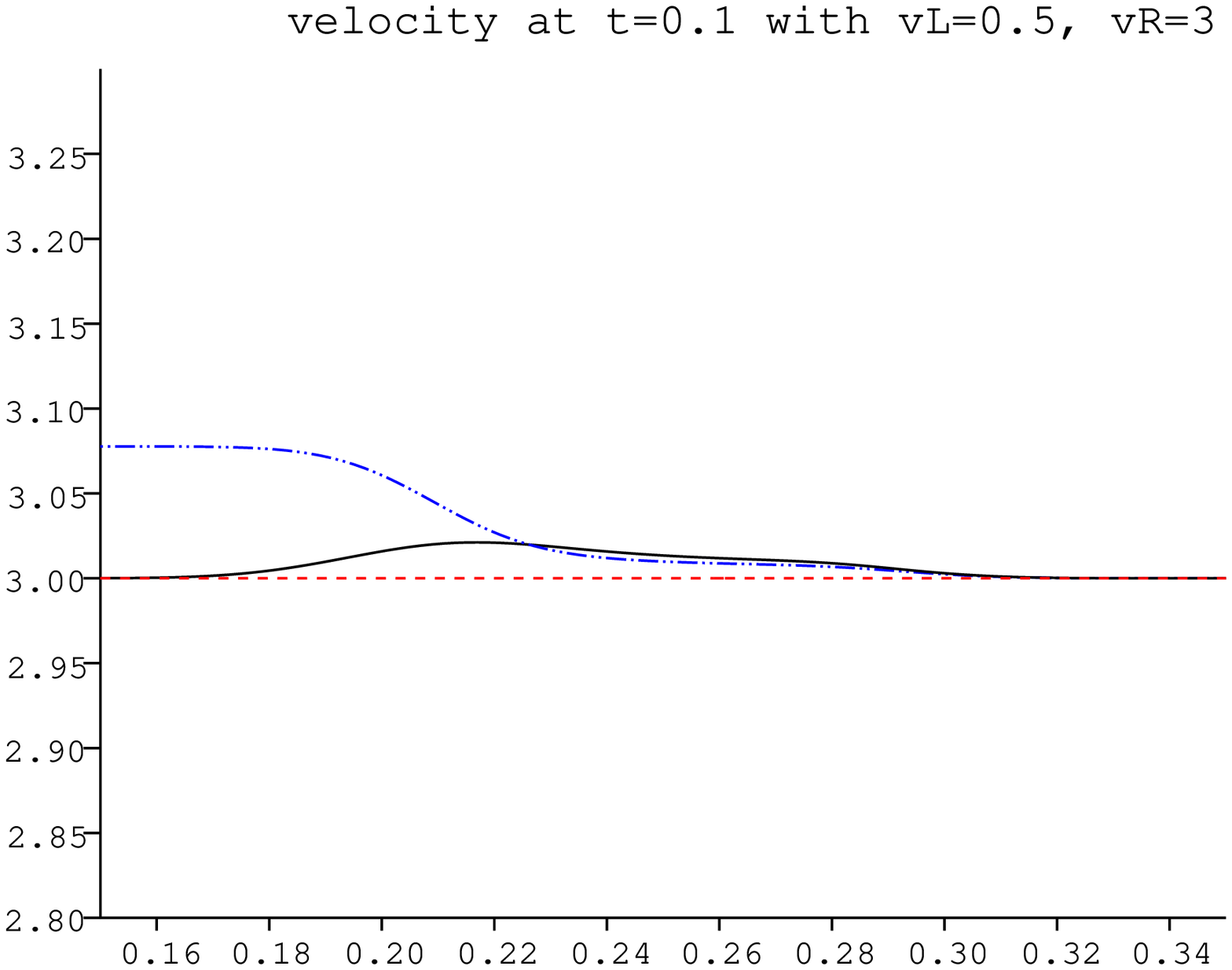}
    \caption{{\bf Test 2b}: Detailed view of a part of the
      computational domain of Figure \ref{rhov-test2b}.
%    Solution (zoomed) of the constrained Riemann solver $\Rsol_{2}^q$
%      with data $\rho^l=4$, $\rho^r=1.5$,
%      $v^l=0.5$, $v^r=3$ and $q=3$: exact solution (dashed line), ghost cell method  (dash-dotted line),
%      our method (continuous line).
}
    \label{rhov-test2b-zoom}
  \end{center}
\end{figure}

\section*{Acknowledgements}

The authors were supported by the NUSMAIN-NOMAIN 2009 project
of the Galileo program 2009 (French-Italian cooperation program).

The authors thank Christophe Chalons for useful
discussions on the numerical section.

%%%%%%%%%%%%%%%%%%%%%%%%%%%%%%%%%%%%%%%%%%%
%
%  BIBLIOGRAPHY
%
%%%%%%%%%%%%%%%%%%%%%%%%%%%%%%%%%%%%%%%%%%%%

{\small{ \bibliographystyle{abbrv}
    
    \bibliography{1.bib} }}

\end{document}